\documentclass[11pt]{article}
\usepackage[T1]{fontenc}
\usepackage[utf8]{inputenc}
\usepackage{lmodern}
\usepackage{amsmath,amssymb,amsthm,mathtools}
\usepackage[margin=2.4cm]{geometry}
\usepackage{booktabs}
\usepackage{parskip}           % no paragraph indentation
\usepackage[backref=page]{hyperref}
\hypersetup{colorlinks=true,linkcolor=blue,urlcolor=blue}

\hypersetup{hidelinks}

\newtheorem{theorem}{Theorem}[section]
\numberwithin{equation}{section}
\newtheorem{proposition}[theorem]{Proposition}
\newtheorem{lemma}[theorem]{Lemma}
\newtheorem{corollary}[theorem]{Corollary}
\newtheorem{introthm}{Theorem}            % main results in the Introduction: A, B, C, D

\newtheorem{introcor}[introthm]{Corollary}
\theoremstyle{definition}
\newtheorem{definition}[theorem]{Definition}
\newtheorem{remark}[theorem]{Remark}

\newcommand{\SLB}{\mathsf{SLB}}
\newcommand{\BR}{\mathcal{BR}}
\newcommand{\YB}{\mathcal{YB}}
\newcommand{\FS}{\mathcal{FS}}
\newcommand{\BLS}{\mathcal{BLS}}
\newcommand{\LYB}{\mathcal{LYB}}

\usepackage{enumerate}          % needed for \begin{enumerate}[(1)]
\newcommand{\selabel}[1]{\label{se:#1}}
\newcommand{\seref}[1]{Section~\ref{se:#1}}
\newcommand{\sselabel}[1]{\label{sse:#1}}

\def\thlabel#1{\label{th:#1}}
\def\thref#1{Theorem~\ref{th:#1}}

\newcommand{\lelabel}[1]{\label{le:#1}}
\newcommand{\leref}[1]{Lemma~\ref{le:#1}}
\newcommand{\prlabel}[1]{\label{pr:#1}}
\newcommand{\prref}[1]{Proposition~\ref{pr:#1}}
\newcommand{\colabel}[1]{\label{co:#1}}
\newcommand{\coref}[1]{Corollary~\ref{co:#1}}
\newcommand{\relabel}[1]{\label{re:#1}}
\newcommand{\reref}[1]{Remark~\ref{re:#1}}

\newcommand{\delabel}[1]{\label{de:#1}}
\newcommand{\deref}[1]{Definition~\ref{de:#1}}
\newcommand{\eqlabel}[1]{\label{eq:#1}}
\newcommand{\equref}[1]{(\ref{eq:#1})}

\newtheorem{remarks}[theorem]{Remarks}

\newcommand{\bZ}{\mathbb{Z}}
\newcommand{\BIP}{\mathcal{BIP}}
\newcommand{\Sym}[1]{S_{#1}}    % symmetric group; change to \Sym{#1} to revert

\title{Bipartite solutions of the set-theoretic Yang–Baxter equation}
\author{G. Militaru}
\date{}

\begin{document}
\maketitle

\begin{abstract}

We introduce a new class of unitary set-theoretic Yang--Baxter solutions, called \emph{bipartite} solutions. This class, which is purely combinatorial in nature, forms a full subcategory of $\mathcal{LYB}$, the category of Lyubashenko type solutions of Agore, Chirvasitu and the author, and is particularly amenable to explicit classification results. We give an intrinsic necessary and sufficient condition for a solution of the Yang--Baxter equation to be isomorphic to a bipartite solution, we settle the isomorphism problem within this class and, as an application, we determine their automorphism groups. If $X$ is a finite set with $|X| = n$, the number of isomorphism classes of bipartite solutions on $X$ is $\lfloor n/2 \rfloor + 1 + \sum_{m=2}^{n} \bigl( o(m, n-m) - 1 \bigr)$, where $o(m, k)$ denotes the number of orbits of the symmetric group $S_m$ acting by conjugation on the $k$-element multisets of elements of $S_m$. The paper also offers a comparative study: for small $|X| = n$ we count, up to isomorphism, all set-theoretic Yang--Baxter solutions on $X$, together with five distinguished families of solutions. Throughout the range we were able to compute, the class $\mathcal{LYB}$ accounts for a large share of all bijective and of all non-degenerate solutions, strikingly close to the total.
% while those coming from skew
%left braces are very few in each of the four tables. For
%$n = 4$, for instance, there are $253$ isomorphism classes of right-left
%non-degenerate YB-solutions; $223$ of them are $\mathcal{LYB}$-solutions, while only $2$
%arise from skew left braces.
\end{abstract}

\medskip


\vspace{.3cm}


\tableofcontents

\section*{Introduction}

The quantum Yang--Baxter equation \cite{baxter, Yang} plays an important role in many areas of mathematics and mathematical physics, such as solvable models of statistical mechanics, conformal field theory, Hopf algebras and quantum groups, invariants of knots and three-manifolds, braided monoidal categories, and non-commutative algebraic geometry (see \cite{K} and the references therein). The study of its solutions at the level of sets was proposed by Drinfel'd in \cite{dr}, and over the last thirty years this problem has received a great deal of attention, in particular because of its connections with combinatorics and graph theory, knot theory, and (semi)group theory. For further details and additional motivation, we refer to the recent monographs \cite{BES, CV} and the references therein.

We denote by $\YB$ the category of all solutions of the set-theoretic Yang--Baxter equation, i.e. its objects are pairs $(X, \, R)$ consisting of a set $X$ and a map $R: X \times X \to X \times X$ such that
\begin{equation}\label{qybe2}
R^{12} R^{13} R^{23} = R^{23}R^{13}R^{12}
\end{equation}
where $R^{ij} : X \times X \times X \to X \times X \times X$ denotes $R$ acting in the $i^{th}$ and $j^{th}$ components. As detailed and motivated in \cite{CMv1}, the category $\YB$ is isomorphic to a variety of algebras, and the problem of classifying the solutions of this equation is a typical problem of universal algebra \cite{buris}, with a pronounced combinatorial and graph-theoretic flavour \cite{acm1, GI04, GIM, Yang16} - after all, functions can be cast as directed graphs.

The classification of all solutions of this equation (usually only bijective maps are
considered) is a wide open problem, and the chances of settling it in full generality are
close to zero. For this reason, various special classes of solutions have been classified
over the years, most of them centred on the bijective non-degenerate solutions introduced
in \cite[Definition 1.1]{etingof}. Although these form a rather restrictive class, even
they have not been classified, except under additional restrictions
\cite[Theorems 2.9 and 2.13]{etingof} or \cite[Theorem 3.1]{etingof2}. The most general
result in this direction is due to Soloviev \cite[Theorems 2.7 and 2.8]{soloviev}, which
parameterizes all bijective non-degenerate solutions by bijective cocycle $7$-tuples. It is curious that this result seems to have gone largely unnoticed in the literature on non-degenerate solutions.

The strategy adopted has been to define various categories and to associate to each of them a special class of solutions. We mention only a few of them: braided operators,  bijective $1$-cocycles \cite{LuYZ}, braided groups \cite{tak_2001}, (bi)racks \cite{PPZ}, cycle sets \cite{rump}, skew left (inverse or weak) braces \cite{CMT, CMMS, vendramin1}, trusses \cite{brz} and post-groups \cite{bai_guo}. More recently, jointly with A.L. Agore and A. Chirvasitu, we introduced three further classes: the first is associated with solutions of Frobenius--Separability type \cite{acm1} (i.e. the variety of $\FS$-solutions below); the second is the variety of P{\l}onka bi-magmas \cite{ACMv3}, which parameterizes all solutions of Baaj--Long--Skandalis type \cite{ACM2} (denoted below as $\BLS$-solutions); and finally, the category of Lyubashenko type solutions (denoted $\LYB$-solutions) \cite[Definition 2.11]{ACM2}.

The category $\LYB$ has a feature which distinguishes it from other categories of YB-solutions: although it is a full subcategory of $\YB$, which is isomorphic to a variety of algebras, $\LYB$ itself is not equivalent to any variety of algebras (\thref{lybnotvar}). Nevertheless, in \seref{comparative} we will show that it is by far the class closest to $\YB$, in the sense that it contains the largest number of isomorphism classes of solutions; this number is surprisingly close to the number of all YB-solutions. However, the inclusion $\LYB \subseteq \YB$ is strict and, even more, there exist YB-solutions which are neither limits nor colimits in $\YB$ of diagrams in $\LYB$ \cite[Corollary 2.14]{ACM2}. This shows how fascinating, and at the same time how unruly, the category $\YB$ of all set-theoretic YB-solutions is.

In \seref{bipartite} we introduce a new class of unitary solutions of the set-theoretic Yang--Baxter equation, which we call \emph{bipartite solutions}; they form a full subcategory $\BIP$, much easier to handle, of the category $\LYB$ of all Lyubashenko type solutions. The underlying set is partitioned into two parts, $X = Y \sqcup Z$, the elements of $Y$ act on $Z$ by \emph{arbitrary} permutations, and two elements of the same part do not interact. More precisely, a \emph{bipartite datum} is a triple $(Y, \, Z, \, g)$, where $Y$ and $Z$ are disjoint sets and $g = (g_y)_{y\in Y}$ is a family of permutations of $Z$. The associated map $R_g \colon X \times X \to X \times X$, where $X := Y \sqcup Z$, is the identity on $(Y \times Y) \cup (Z \times Z)$ and
\begin{equation*}
R_g (y, \, z) := \bigl( g_y (z), \, y \bigr), \qquad R_g (z, \, y) := \bigl( y, \, g_y^{-1} (z) \bigr), \qquad \text{for all} \,\, y \in Y, \,\, z\in Z.
\end{equation*}
The datum $(Y, \, Z, \, g)$ is called \emph{split} if $g_y = {\rm Id}_Z$, for all $y\in Y$. The main results of the paper are the following.

\begin{introthm}[\thref{Rg}]
For any bipartite datum $(Y, \, Z, \, g)$, the map $R_g$ is a solution of the braid equation with $R_g^2 = {\rm Id}_{X\times X}$. Equivalently, $\tau_X \circ R_g$ is a unitary YB-solution; it is of diagonal type and hence of Lyubashenko type.
\end{introthm}

Unlike $\YB$, $\BLS$ and $\FS$, neither $\BIP$ nor $\LYB$ is equivalent to a variety of algebras (\prref{BIPnotvariety} and \thref{lybnotvar}). The bipartite solutions can be recognized directly in terms of the map $R$, without any reference to a bipartite datum:

\begin{introthm}[\thref{recognition}]
Let $R \colon X \times X \to X \times X$ be a map and ${\rm Fix} (R) := \{ (x, \, x') \in X \times X \, \mid \, R(x, \, x') = (x, \, x') \}$. Then $(X, \, R)$ is isomorphic to a bipartite solution if and only if $R^2 = {\rm Id}_{X \times X}$, ${\rm Fix} (R)$ is an equivalence relation on $X$ with at most two classes and, if there are two classes, one of them, say $Y$, satisfies $R (y, \, x) \in X \times \{y\}$ and $R (x, \, y) \in \{y\} \times X$, for all $y\in Y$ and $x \in X \setminus Y$.
\end{introthm}

For a non-split datum, the partition $X = Y \sqcup Z$ can be recovered from $R_g$ alone (\leref{intrinsic}), and this allows us to solve the isomorphism problem (\thref{bipiso}): a split bipartite solution is never isomorphic to a non-split one, while two non-split bipartite solutions $R_g$ and $R_{g'}$ are isomorphic if and only if $|Z| = |Z'|$ and, after identifying $Z$ and $Z'$ by a bijection, the multisets of values of $g$ and $g'$ are conjugate in $\Sym{Z}$. \thref{bipiso} classifies the bipartite solutions up to isomorphism; in particular, it allows us to count them on a finite set:

\begin{introcor}[\coref{countbip}]
If $|X| = n$, then the number of isomorphism classes of bipartite solutions on $X$ is
\begin{equation*}
\mathrm{bip} (n) = \Bigl\lfloor \frac{n}{2} \Bigr\rfloor + 1 + \sum_{m=2}^{n} \bigl( o(m, \, n-m) - 1 \bigr),
\end{equation*}
where $o(m, \, k)$ is the number of orbits of the action of $\Sym{m}$ by conjugation on the set of multisets of cardinality $k$ of elements of $\Sym{m}$; the numbers $o(m, \, k)$ can be computed explicitly by means of the Cauchy--Frobenius lemma.
\end{introcor}

For $n = 1, \dots, 8$ we obtain $\mathrm{bip} (n) = 1$, $2$, $3$, $7$, $17$, $57$, $305$, $4120$ (Table~\ref{tab:bipcount}). On the other hand, for $|X| \geq 3$ no bipartite solution is isomorphic to a solution arising from a skew left brace (\coref{bipnotslb}).

As a second important application of \thref{bipiso}, we compute the automorphism groups of the bipartite solutions: they are semidirect products. For $h \in G_g := \{ g_y \, | \, y \in Y\}$ we denote by $Y_h := \{ y\in Y \, | \, g_y = h\}$ the fibres of $g$.

\begin{introthm}[\thref{autbip}]
Let $(Y, \, Z, \, g)$ be a bipartite datum and $X = Y \sqcup Z$.
\begin{enumerate}[(1)]
\item If the datum is split, then ${\rm Aut} (X, \, R_g) = \Sym{Y} \times \Sym{Z}$ if $|Y| \neq |Z|$, and ${\rm Aut} (X, \, R_g) \cong (\Sym{Y} \times \Sym{Z}) \rtimes \bZ/2$ if $|Y| = |Z| \geq 1$, where, after identifying $Y$ and $Z$ by a bijection, the generator of $\bZ/2$ interchanges the two factors.
\item If the datum is non-split, then
\begin{equation*}
{\rm Aut} (X, \, R_g) \cong \Bigl( \prod_{h \in G_g} \Sym{Y_h} \Bigr) \rtimes N_g,
\end{equation*}
where $N_g$ is the subgroup of $\Sym{Z}$ consisting of all $\beta$ such that $\beta \, G_g \, \beta^{-1} = G_g$ and $|Y_{\beta h \beta^{-1}}| = |Y_h|$, for all $h \in G_g$, acting on the direct product by mapping each factor $\Sym{Y_h}$ isomorphically onto $\Sym{Y_{\beta h \beta^{-1}}}$. In particular, if $y \mapsto g_y$ is injective, then ${\rm Aut} (X, \, R_g)$ is isomorphic to the normalizer in $\Sym{Z}$ of the set $G_g$.
\end{enumerate}
\end{introthm}

Moreover, the bipartite solutions are the key ingredient in \cite{MilAut}, where it is proved that the category of bijective (in fact, unitary) YB-solutions is \emph{universal} and \emph{finitely universal} 
in the sense of Babai \cite{Babai1995}, i.e. every group is isomorphic to the automorphism group of a unitary YB-solution, which can be chosen on a finite set if the group is finite.

Since there are a great many categories from which YB-solutions can be constructed, a natural question that arises in this context is the following:

\emph{How much of the category} $\YB$ \emph{do these types of solutions, associated with various categories, actually cover}?

In Section~\ref{se:comparative} we present four tables, obtained by means of programs written in C and Python which enumerate the solutions exhaustively on small sets and sort them into isomorphism classes.
More precisely, we classify all YB-solutions, as well as the five distinguished families of solutions mentioned above, in four settings: all solutions (Table~\ref{tab:unrestricted}), bijective solutions (Table~\ref{tab:bijective}), unitary solutions (Table~\ref{tab:unitary}) and, finally, right-left non-degenerate solutions (Table~\ref{tab:main}). We hope that they offer a relevant picture of each class of solutions; this picture is not the one usually advertised. Perhaps surprisingly to some readers, the tables show that the solutions arising from skew left braces are very scarce, even among the right-left non-degenerate ones.\footnote{It is frequently stated in the literature that \emph{every non-degenerate YB-solution arises from a skew left brace}. In the strict sense of being isomorphic to the solution induced by a skew left brace on its whole underlying set, this statement fails (see \cite[Proposition 3.2]{ACM2}), and the tables in Section~\ref{se:comparative} show that this is far from being the case.} Moreover, throughout the range we were able to compute, the class $\LYB$ accounts for a large share of all bijective and of all right-left non-degenerate solutions, strikingly close to the total. For $n = 4$, for instance, there are $253$ isomorphism classes of right-left non-degenerate YB-solutions; $223$ of them are of Lyubashenko type, while only $2$ arise from skew left braces.

These numbers show that we are still very far from a solution of Drinfel'd's problem \cite{dr}. What can be said with certainty, however, is that the Yang--Baxter equation is not essentially a problem about groups, no matter how many group structures we put on one and the same set. In this context, in \seref{conclusion} we propose three new directions of research which may help us to better understand the unruly category $\YB$; they complement Problems A, B and C formulated in \cite[Introduction and Section 4]{ACM2}.

\section{Preliminaries}\selabel{prel}

Given a set $X$, we denote by ${\rm Id}_X$ the identity map of $X$, by $\Sym{X}$ the group of permutations of $X$ and by $\tau_X \colon X\times X \to X \times X$ the flip map, $\tau_X (x, \, y) = (y, \, x)$, for all $x$, $y\in X$. If $R\colon X \times X \to X\times X$ is a map, we define:
\begin{equation*}
R^{12} := R \times {\rm Id}_X, \qquad R^{23} := {\rm Id}_X \times R, \qquad R^{13} := ({\rm Id}_X \times \tau_X) \circ R^{12} \circ ({\rm Id}_X \times \tau_X).
\end{equation*}
Recall that $R$ is called a set-theoretic solution of the \emph{Yang--Baxter equation}, or briefly a \emph{YB-solution} on $X$, if (\ref{qybe2}) holds, and a solution of the \emph{braid equation} if
\begin{equation}\eqlabel{braid}
R^{12} R^{23} R^{12} = R^{23} R^{12} R^{23}.
\end{equation}
It is well known and easy to see that $R$ is a solution of the braid equation if and only if $\tau_X \circ R$ is a YB-solution. A map $R \colon X\times X \to X\times X$ is called \emph{involutive} if $R^2 = {\rm Id}_{X \times X}$, \emph{unitary} if $R^{21} R = {\rm Id}_{X \times X}$, where $R^{21} := \tau_X R \tau_X$, and of \emph{diagonal type} if $R(x, \, x) = (x, \, x)$, for all $x \in X$.

For the basic concepts of universal algebras we refer to \cite{adamek_ros, bergman} and for category theory to \cite{Ago}. We only recall, for future use, a fundamental theorem of Birkhoff \cite[Theorem 3.9]{adamek_ros}:
a class of universal algebras of the same signature is a variety of algebras if and only if it is closed under subalgebras, direct products and homomorphic images. Any variety of algebras $\mathcal{V}$ is (small) complete and cocomplete, i.e. it has all (small) limits and colimits \cite[p. 18]{bergman}. For other properties of varieties of algebras, see \cite[Remark 3.6 and Corollary 3.7]{adamek_ros}; in particular, any variety of algebras $\mathcal{V}$ is a \emph{locally finitely presentable category} \cite[Definition 1.9]{adamek_ros}.

Throughout, $\YB$ denotes the category of YB-solutions: its objects are pairs $(X, \, R)$, where $X$ is a set and $R$ a YB-solution on $X$, and a morphism $\sigma \colon (X, \, R) \to (X', \, R')$ in $\YB$ is a map $\sigma \colon X \to X'$ such that $(\sigma \times \sigma) \circ R = R' \circ (\sigma \times \sigma)$. Two solutions are called \emph{isomorphic} if they are isomorphic as objects of $\YB$, and ${\rm Aut} \, (X, \, R)$ denotes the automorphism group of $(X, \, R)$ in $\YB$. In the same fashion, $\BR$ denotes the category of solutions of the braid equation, with the same definition of morphisms. Since $(\sigma \times \sigma) \circ \tau_X = \tau_{X'} \circ (\sigma \times \sigma)$ for any map $\sigma \colon X \to X'$, the functor
\begin{equation}\eqlabel{iso_br_yb}
\BR \to \YB, \qquad (X, \, R) \mapsto (X, \, \tau_X \circ R), \qquad \sigma \mapsto \sigma
\end{equation}
is an isomorphism of categories (\cite[Section 1]{ACM2}). In particular,
\begin{equation}\eqlabel{auttau}
{\rm Aut} \, (X, \, R) = {\rm Aut} \, (X, \, \tau_X \circ R)
\end{equation}
and a bijection $\sigma \colon X \to X'$ is an isomorphism $(X, \, R) \to (X', \, R')$ if and only if it is an isomorphism $(X, \, \tau_X \circ R) \to (X', \, \tau_{X'} \circ R')$. All the classes of solutions of the braid equation considered in this paper (FS-solutions, solutions arising from skew left braces and bipartite solutions) will be compared with classes of YB-solutions by means of the isomorphism \equref{iso_br_yb}; this does not affect isomorphism classes or automorphism groups.

\textbf{Bi-magmas and Yang--Baxter bi-magmas.} A \emph{bi-magma} is a triple $(X, \, \cdot, \, \ast)$ consisting of a set $X$ and two binary operations $\cdot$, $\ast \colon X \times X \to X$. As explained in \cite[Introduction]{ACMv3}, there is a bijective correspondence between maps $R \colon X\times X \to X \times X$ and bi-magma structures on $X$: throughout, the map $R = R_{(\cdot, \, \ast)}$ associated with the bi-magma $(X, \, \cdot, \, \ast)$ is given by
\begin{equation}\eqlabel{notatieR}
R(x, \, y) = R_{(\cdot, \, \ast)} \, (x, \, y) = (x\cdot y, \, x \ast y), \qquad \text{for all} \,\, x, \, y \in X.
\end{equation}
A morphism of bi-magmas $\sigma \colon (X, \, \cdot, \, \ast) \to (X', \, \cdot', \, \ast')$ is a map $\sigma \colon X \to X'$ such that $\sigma (x \cdot y) = \sigma(x) \cdot' \sigma(y)$ and $\sigma (x \ast y) = \sigma(x) \ast' \sigma(y)$, for all $x$, $y\in X$. Thus a map $\sigma$ is a morphism of solutions $(X, \, R_{(\cdot, \, \ast)}) \to (X', \, R_{(\cdot', \, \ast')})$ if and only if it is a morphism between the associated bi-magmas. Furthermore, $R = R_{(\cdot, \, \ast)}$ is a YB-solution if and only if the following compatibilities hold for all $x$, $y$, $z\in X$:
\begin{align}
(x\cdot y) \cdot z &= \bigl(x \cdot (y \ast z) \bigr) \cdot (y \cdot z), \eqlabel{YB1}\\
\bigl( x \cdot (y \ast z) \bigr) \ast (y\cdot z ) &= (x \ast y) \cdot \bigl( (x\cdot y) \ast z \bigr), \eqlabel{YB2}\\
x \ast (y \ast z) &= (x \ast y) \ast \bigl( (x\cdot y) \ast z\bigr). \eqlabel{YB3}
\end{align}
A bi-magma satisfying \equref{YB1}-\equref{YB3} is called a \emph{Yang--Baxter bi-magma} \cite[Definition 1.4]{ACMv3}. The Yang--Baxter bi-magmas form a variety of algebras ${}_{\rm YB}{\rm BiMag}$ and the functor
\begin{equation*}
{}_{\rm YB}{\rm BiMag} \to \YB, \qquad (X, \, \cdot, \, \ast) \mapsto (X, \, R_{(\cdot, \, \ast)})
\end{equation*}
is an isomorphism of categories. In other words, the classification of all YB-solutions is a problem of universal algebra, as already pointed out in \cite{CMv1}.

\textbf{Non-degeneracy.} For a bi-magma $(X, \, \cdot, \, \ast)$ and $a \in X$ we consider the translation maps $l_a$, $r_a$, $\rho_a$, $\lambda_a \colon X \to X$ defined by
\begin{equation}\eqlabel{transl}
l_a (x) := a \cdot x, \qquad r_a (x) := x \ast a, \qquad \rho_a (x) := x \cdot a, \qquad \lambda_a (x) := a \ast x,
\end{equation}
for all $x\in X$. The map $R = R_{(\cdot, \, \ast)}$ is called \emph{left-right non-degenerate} if $l_a$ and $r_a$ are bijective for all $a\in X$, i.e. $(X, \, \cdot)$ is a left quasi-group and $(X, \, \ast)$ is a right quasi-group. This is the classical notion of non-degeneracy introduced in \cite[Definition 1.1]{etingof} for solutions of the braid equation (see \cite[Remark 1.3]{ACMv3}). Similarly, $R$ is called \emph{right-left non-degenerate} if $\rho_a$ and $\lambda_a$ are bijective for all $a\in X$, i.e. $(X, \, \cdot)$ is a right quasi-group and $(X, \, \ast)$ is a left quasi-group; this is the natural counterpart of non-degeneracy for YB-solutions \cite[Remark 1.2]{ACM2}. Since $\tau_X \circ R \, (x, \, y) = (x \ast y, \, x \cdot y)$, the right translations of the first operation and the left translations of the second operation of $\tau_X \circ R$ are precisely the maps $r_a$ and $l_a$ of $R$. Hence, for any map $R \colon X \times X \to X \times X$ we have:
\begin{equation}\eqlabel{dict}
R \,\, \text{is left-right non-degenerate} \,\, \Longleftrightarrow \,\, \tau_X \circ R \,\, \text{is right-left non-degenerate}.
\end{equation}

\textbf{BLS-solutions and P{\l}onka bi-magmas.} A map $R \colon X\times X \to X\times X$ is called \emph{commutative} (resp. \emph{cocommutative}) if $R^{12} R^{13} = R^{13} R^{12}$ (resp. $R^{13} R^{23} = R^{23} R^{13}$), and a solution of the set-theoretic \emph{Long equation} if $R^{12} R^{23} = R^{23} R^{12}$. A \emph{BLS-solution} (from Baaj, Long and Skandalis) \cite[Definition 2.1]{ACM2} is a commutative and cocommutative solution of the Long equation. Any BLS-solution is a YB-solution \cite[Remarks 2.2 (3)]{ACM2}; we denote by $\BLS$ the full subcategory of $\YB$ consisting of all BLS-solutions. Recall from \cite[Definition 3.8]{ACMv3} that a \emph{P{\l}onka bi-magma} is a bi-magma $(X, \, \cdot, \, \ast)$ such that for all $x$, $y$, $z\in X$:
\begin{align*}
& (x \cdot y) \cdot z = (x \cdot z) \cdot y, \qquad x \cdot (y \cdot z) = x \cdot y, \qquad x \ast (y \ast z) = y \ast (x \ast z), \qquad (x \ast y) \ast z = y \ast z, \\
& x \ast (y \cdot z) = (x \ast y) \cdot z, \qquad (x \cdot z) \ast y = x \ast y, \qquad x \cdot (y\ast z) = x \cdot z,
\end{align*}
i.e. $(X, \, \cdot)$ is a right P{\l}onka magma, $(X, \, \ast)$ is a left P{\l}onka magma and the last three compatibility conditions hold. The P{\l}onka bi-magmas form a variety of algebras denoted by ${\rm PBiMag}$.

%We recall from \cite{ACMv3, ACM2} the following results:
%\begin{theorem}\thlabel{BLSrecall}
%\begin{enumerate}[(1)]
%\item \textup{(\cite[Theorem 2.4]{ACM2})} A map $R = R_{(\cdot, \, \ast)}$ is a BLS-solution if and only if $(X, \, \cdot, \, \ast)$ is a P{\l}onka bi-magma. Moreover, the functor ${\rm PBiMag} \to \BLS$, $(X, \, \cdot, \, %\ast) \mapsto (X, \, R_{(\cdot, \, \ast)})$ is an isomorphism of categories.
%\item \textup{(\cite[Theorem 3.25]{ACMv3}, \cite[Theorem 2.6]{ACM2})} A map $R \colon X\times X \to X\times X$ is a BLS-solution if and only if there exist a partition $X = \coprod_{i\in I} \, X_i$ and two families $\bigl(f_i^j \bigr)_{(i, j) \in I\times I}$, $\bigl(g_i^j \bigr)_{(i, j) \in I\times I}$ of maps $f_i^j$, $g_i^j \colon X_i \to X_i$ such that
%\begin{equation*}
%f_i^j \circ f_i^k = f_i^k \circ f_i^j, \qquad g_i^j \circ g_i^k = g_i^k \circ g_i^j, \qquad g_i^j \circ f_i^k = f_i^k \circ g_i^j
%\end{equation*}
%for all $i$, $j$, $k\in I$, and $R(x_i, \, x_j) = \bigl( f_i^j (x_i), \, g_j^i (x_j) \bigr)$, for all $x_i \in X_i$, $x_j \in X_j$ and $i$, $j\in I$.
%\item \textup{(\cite[Theorem 3.28]{ACMv3})} A BLS-solution $R$ is right-left non-degenerate if and only if it is bijective, if and only if all the maps $f_i^j$, $g_i^j$ of a decomposition as in \textup{(2)} can be chosen to be bijective. Furthermore, $R$ is left-right non-degenerate if and only if $|X| = 1$.
%\end{enumerate}
%\end{theorem}

\textbf{Lyubashenko type solutions.} The basic example of a YB-solution, attributed to Lyubashenko in \cite[Example 1]{dr}, is the following: if $f$, $g \colon X \to X$ are two commuting maps, then $R_{(f, \, g)} (x, \, y) := \bigl( f(x), \, g(y) \bigr)$ is a YB-solution on $X$. The following generalization was introduced in \cite{ACM2}.

\begin{definition}[{\cite[Definition 2.10]{ACM2}}]\delabel{lymag}
A YB-solution $(X, \, R)$ is called of \emph{Lyubashenko type} if there exist a partition $X = \coprod_{i\in I} X_i$ and two families $(f_i)_{i\in I}$, $(g_i)_{i\in I}$ of maps $f_i$, $g_i \colon X_i \to X_i$ such that for any $i \in I$ and $x$, $y\in X_i$:
\begin{equation}\eqlabel{LYB}
g_i \circ f_i = f_i \circ g_i, \qquad R(x, \, y) = \bigl( f_i(x), \, g_i(y) \bigr).
\end{equation}
We denote by $\LYB$ the full subcategory of $\YB$ consisting of all Lyubashenko type solutions.
\end{definition}

Note that \equref{LYB} only concerns the restriction of $R$ to the blocks $X_i \times X_i$: nothing is required for $R(x, \, y)$ if $x$ and $y$ belong to different blocks. We recall that
any YB-solution of diagonal type is of Lyubashenko type, for the partition of $X$ into singletons, and any BLS-solution is of Lyubashenko type. We thus have $\BLS \subseteq \LYB \subseteq \YB$,
and both inclusions are strict \cite[Remarks 2.11, Example 2.12]{ACM2}. Unlike the categories $\YB$ and $\BLS$, the category $\LYB$ is not equivalent to a variety of algebras:
this is proved in \thref{lybnotvar} below.

\textbf{FS-solutions and Kimura semigroups.} A map $R\colon X\times X \to X\times X$ is called a solution of the set-theoretic \emph{Frobenius-Separability equation}, or briefly an \emph{FS-solution}, if \cite[Definition 2.1]{acm1}
\begin{equation*}
R^{12} R^{23} = R^{23} R^{13} = R^{13} R^{12}.
\end{equation*}
Any FS-solution is a solution of the braid equation. By \cite[Proposition 2.4]{acm1}, $R = R_{(\cdot, \, \ast)}$ is an FS-solution if and only if for all $x$, $y$, $z\in X$:
\begin{align}
x\cdot(y\cdot z) &= (x\cdot y)\cdot z = x\cdot z, \eqlabel{FS1}\\
x \ast (y\cdot z) &= y\cdot(x \ast z) = x \ast y, \eqlabel{FS2}\\
(x\cdot y) \ast z &= y \ast (x \ast z) = y \ast z. \eqlabel{FS3}
\end{align}
The first condition says that $(X, \, \cdot)$ is a \emph{Kimura semigroup} \cite[Definition 1.5]{acm1}. For our purposes we need two classification results from \cite{acm1}. First, by \cite[Theorem 2.20, Corollary 2.22]{acm1}, any bijective FS-solution is isomorphic to a solution of the form
\begin{equation}\eqlabel{FSbij}
R_\omega \colon (A \times B)^2 \to (A \times B)^2, \qquad R_\omega \bigl( (a_1, \, b_1), \, (a_2, \, b_2) \bigr) = \bigl( (a_1, \, b_2), \, (a_2, \, \omega(b_1)) \bigr),
\end{equation}
for two non-empty sets $A$, $B$ and a permutation $\omega \in \Sym{B}$; two such solutions are isomorphic if and only if the corresponding sets have the same cardinalities and the permutations are conjugate. Consequently, if $|X| = n$, then the number of isomorphism classes of bijective FS-solutions on $X$ is $\sum_{d \mid n} p(d)$, where $p$ is the partition function. Secondly, by \cite[Theorem 2.19]{acm1} and \cite[Remarks 2.21 (2)]{acm1}, the left-right non-degenerate FS-solutions on $X$ are precisely the maps $R_f (x, \, y) = \bigl(y, \, f(x)\bigr)$, where $f \in \Sym{X}$; two of them are isomorphic if and only if the corresponding permutations are conjugate, so there are $p(n)$ isomorphism classes of them if $|X| = n$. Being solutions of the braid equation, FS-solutions are regarded as YB-solutions via \equref{iso_br_yb}: we denote by $\FS$ the full subcategory of $\YB$ consisting of all YB-solutions of the form $\tau_X \circ R$, where $R$ is an FS-solution.

\textbf{Skew left braces.} Recall \cite[Definition 1.1]{smo_ven} that a \emph{skew left brace} is a bi-magma $(X, \, \cdot, \, \ast)$ such that $(X, \, \cdot)$ and $(X, \, \ast)$ are groups and $x \ast (y\cdot z) = (x \ast y) \cdot x^{-1} \cdot (x \ast z)$, for all $x$, $y$, $z \in X$, where $x^{-1}$ denotes the inverse of $x$ in $(X, \, \cdot)$ and $\overline{x}$ the inverse of $x$ in $(X, \, \ast)$. If $B = (X, \, \cdot, \, \ast)$ is a skew left brace, then the map
\begin{equation}\eqlabel{slbsol}
R_B \colon X \times X \to X \times X, \qquad R_B (x, \, y) := \bigl( x^{-1} \cdot (x\ast y), \, \overline{x^{-1} \cdot (x\ast y)} \ast x \ast y \bigr)
\end{equation}
is a bijective and left-right non-degenerate solution of the braid equation \cite[Theorem 3.1]{vendramin1}, and $B \mapsto (X, \, R_B)$ defines a functor from the category of skew left braces to $\BR$ (see \cite[Section 3]{ACM2}). By \equref{dict}, $\tau_X \circ R_B$ is a bijective and right-left non-degenerate YB-solution. We denote by $\SLB$ the class of all YB-solutions isomorphic to $\tau_X \circ R_B$, for some skew left brace $B$.

\section{Bipartite Yang-Baxter solutions}\selabel{bipartite}
In this section we introduce a new class of bijective solutions
of the braid equation which we call \emph{bipartite}. This family is rather flexible and user-friendly for the purpose of
classification and describing structure. Furthermore, via the isomorphism \equref{iso_br_yb}, it gives rise to a full subcategory of the
category $\LYB$ of all Lyubashenko type solutions (see \thref{Rg} and \deref{BIPcat}).

\textbf{The bipartite solution of a family of permutations.} Let $Y$ and $Z$ be two disjoint sets
and let $g \colon Y \to \Sym{Z}$, $y \mapsto g_y$, be an arbitrary map, i.e. a family
$(g_y)_{y\in Y}$ of permutations of $Z$; no relation between the permutations $g_y$ is assumed. We
consider the set $X := Y \sqcup Z$ and the map $R_g \colon X \times X \to X \times X$ given by
\begin{equation}\eqlabel{Rg}
R_g (x, \, x') := \begin{cases} (x, \, x'), & \text{if} \,\,\, x, \, x' \in Y \,\,\, \text{or} \,\,\, x, \, x'\in Z,\\
\bigl( g_y (z), \, y \bigr), & \text{if} \,\,\, x = y \in Y, \,\, x' = z \in Z,\\
\bigl( y, \, g_y^{-1} (z) \bigr), & \text{if} \,\,\, x = z \in Z, \,\, x' = y \in Y.
\end{cases}
\end{equation}
Hence, two elements of $Y$, or two elements of $Z$, do not interact; an element $y\in Y$ placed
to the left of an element $z\in Z$ passes over it to the right and acts on it by $g_y$, while an
element $y$ placed to the right of $z$ passes over it to the left and acts on it by $g_y^{-1}$. In
terms of the bi-magma $(X, \, \cdot, \, \ast)$ associated to $R_g$ by \equref{notatieR}, for all
$y$, $y'\in Y$ and $z$, $z'\in Z$:
\begin{equation}\eqlabel{bimagRg}
\begin{aligned}
& y \cdot y' = y, \qquad y \ast y' = y', \qquad z \cdot z' = z, \qquad z \ast z' = z',\\
& y \cdot z = g_y (z), \qquad y \ast z = y, \qquad z \cdot y = y, \qquad z \ast y = g_y^{-1} (z).
\end{aligned}
\end{equation}

\begin{definition}\delabel{bipartite}
A triple $(Y, \, Z, \, g)$ consisting of two disjoint sets $Y$, $Z$ and a family
$g = (g_y)_{y\in Y}$ of permutations of $Z$ will be called a \emph{bipartite datum}; the solution
$R_g$ of \equref{Rg}, as well as the YB-solution $\tau_X \circ R_g$, will be called the
\emph{bipartite solution} associated to it. The datum, and the solution, are called \emph{split} if $g_y = {\rm Id}_Z$
for all $y\in Y$ and \emph{non-split} otherwise. We denote by $\BIP$ the class of all bipartite solutions.
\end{definition}

For a split datum, $R_g$ is the partial flip
$\tau_{Y, Z}$ of the partition $X = Y \sqcup Z$: it is the identity on $Y \times Y$ and on
$Z \times Z$, and the flip on $Y \times Z$ and on $Z \times Y$. Furthermore, $R_g = {\rm Id}_{X\times X}$ if $Y = \emptyset$ or
$Z = \emptyset$.

\begin{theorem}\thlabel{Rg}
Let $(Y, \, Z, \, g)$ be a bipartite datum and $X = Y \sqcup Z$. Then:
\begin{enumerate}[(1)]
\item $R_g$ is a bijective solution of the braid equation and $R_g^2 = {\rm Id}_{X \times X}$;
\item $\tau_X \circ R_g$ is a unitary YB-solution of diagonal type. In particular, $(X, \, \tau_X \circ R_g)$ is a Lyubashenko type solution.
\end{enumerate}
\end{theorem}

\begin{proof}
(1) Since $R_g$ is the identity on $Y \times Y$ and on $Z \times Z$, we only have to compute $R_g^2$ on the pairs $(y, \, z)$ and $(z, \, y)$, where $y\in Y$ and $z\in Z$:
\begin{align*}
R_g^2 (y, \, z) &= R_g \bigl( g_y (z), \, y \bigr) = \bigl( y, \, g_y^{-1} g_y (z) \bigr) = (y, \, z),\\
R_g^2 (z, \, y) &= R_g \bigl( y, \, g_y^{-1} (z) \bigr) = \bigl( g_y g_y^{-1} (z), \, y \bigr) = (z, \, y).
\end{align*}
Thus $R_g^2 = {\rm Id}_{X \times X}$; in particular, $R_g$ is bijective and $R_g^{-1} = R_g$.

We prove now that $R_g$ is a solution of the braid equation. First of all, we reduce the number of cases to be checked by a symmetry argument. Let $\omega \colon X^3 \to X^3$, $\omega (x_1, \, x_2, \, x_3) := (x_3, \, x_2, \, x_1)$. For any map $R \colon X\times X \to X \times X$ we have $\omega R^{12} \omega = S^{23}$ and $\omega R^{23} \omega = S^{12}$, where $S := \tau_X R \tau_X$. Hence, $R$ satisfies the braid equation \equref{braid} on a triple $\omega (t)$ if and only if $S$ satisfies \equref{braid} on the triple $t$. On the other hand, it follows from \equref{Rg} that $\tau_X R_g \tau_X = R_{\bar{g}}$, where $\bar{g} := (g_y^{-1})_{y\in Y}$ is again a family of permutations of $Z$. Since $\omega$ maps the triples of the form $(z, \, y, \, y')$ and $(z, \, z', \, y)$ onto triples of the form $(y', \, y, \, z)$ and $(y, \, z', \, z)$, respectively, it is enough to check \equref{braid}, for an arbitrary family $g$, on the triples of the form $(y, \, y', \, z)$, $(y, \, z, \, y')$, $(y, \, z, \, z')$ and $(z, \, y, \, z')$, where $y$, $y'\in Y$ and $z$, $z'\in Z$; if all the entries of a triple belong to $Y$, or all of them belong to $Z$, both sides of \equref{braid} are the identity. Writing $R := R_g$, the two sides of \equref{braid} are computed as follows:
\begin{align*}
(y, \, y', \, z) &\xmapsto{R^{12}} (y, \, y', \, z) \xmapsto{R^{23}} \bigl( y, \, g_{y'} (z), \, y' \bigr) \xmapsto{R^{12}} \bigl( g_y g_{y'} (z), \, y, \, y' \bigr),\\
(y, \, y', \, z) &\xmapsto{R^{23}} \bigl( y, \, g_{y'} (z), \, y' \bigr) \xmapsto{R^{12}} \bigl( g_y g_{y'} (z), \, y, \, y' \bigr) \xmapsto{R^{23}} \bigl( g_y g_{y'} (z), \, y, \, y' \bigr);\\[6pt]
(y, \, z, \, y') &\xmapsto{R^{12}} \bigl( g_y (z), \, y, \, y' \bigr) \xmapsto{R^{23}} \bigl( g_y (z), \, y, \, y' \bigr) \xmapsto{R^{12}} (y, \, z, \, y'),\\
(y, \, z, \, y') &\xmapsto{R^{23}} \bigl( y, \, y', \, g_{y'}^{-1} (z) \bigr) \xmapsto{R^{12}} \bigl( y, \, y', \, g_{y'}^{-1} (z) \bigr) \xmapsto{R^{23}} (y, \, z, \, y');\\[6pt]
(y, \, z, \, z') &\xmapsto{R^{12}} \bigl( g_y (z), \, y, \, z' \bigr) \xmapsto{R^{23}} \bigl( g_y (z), \, g_y (z'), \, y \bigr) \xmapsto{R^{12}} \bigl( g_y (z), \, g_y (z'), \, y \bigr),\\
(y, \, z, \, z') &\xmapsto{R^{23}} (y, \, z, \, z') \xmapsto{R^{12}} \bigl( g_y (z), \, y, \, z' \bigr) \xmapsto{R^{23}} \bigl( g_y (z), \, g_y (z'), \, y \bigr);\\[6pt]
(z, \, y, \, z') &\xmapsto{R^{12}} \bigl( y, \, g_y^{-1} (z), \, z' \bigr) \xmapsto{R^{23}} \bigl( y, \, g_y^{-1} (z), \, z' \bigr) \xmapsto{R^{12}} (z, \, y, \, z'),\\
(z, \, y, \, z') &\xmapsto{R^{23}} \bigl( z, \, g_y (z'), \, y \bigr) \xmapsto{R^{12}} \bigl( z, \, g_y (z'), \, y \bigr) \xmapsto{R^{23}} (z, \, y, \, z').
\end{align*}
In each of the four cases the two sides coincide, and hence $R_g$ is a solution of the braid equation.

(2) By (1) and \seref{prel}, $r := \tau_X \circ R_g$ is a YB-solution. Since $r^{21} = \tau_X r \tau_X = R_g \tau_X$, we obtain $r^{21} r = R_g \tau_X \tau_X R_g = R_g^2 = {\rm Id}_{X \times X}$, i.e. $r$ is unitary. Moreover, for any $x\in X$ the two entries of $(x, \, x)$ lie in the same part of $X$, so $R_g (x, \, x) = (x, \, x)$ and therefore $r (x, \, x) = (x, \, x)$. Finally, any YB-solution of diagonal type is of Lyubashenko type \cite[Remarks 2.11 (1)]{ACM2}.
\end{proof}

Note that in the case $(y, \, y', \, z)$ the composition $g_y g_{y'}$ appears on both sides in the same order; this is the reason why no relation (such as $g_y g_{y'} = g_{y'} g_y$) between the permutations $g_y$ is needed. This is in contrast with the classical Lyubashenko solutions $R(x, \, x') = (f(x'), \, h(x))$, which are solutions of the braid equation only if $f h = h f$.

There are two natural candidates for the morphisms between bipartite
solutions: those of the ambient category of solutions, and those of the underlying bipartite
data. We adopt the first one, and compare it with the second one below.

\begin{definition}\delabel{BIPcat}
We denote by $\BIP$ the full subcategory of $\YB$ whose objects are the YB-solutions isomorphic
to $(X, \, \tau_X \circ R_g)$ for some bipartite datum $(Y, \, Z, \, g)$. Thus a \emph{morphism of
bipartite solutions} $\sigma \colon (X, \, \tau_X \circ R_g) \to (X', \, \tau_{X'} \circ R_{g'})$
is an arbitrary map $\sigma \colon X \to X'$ such that
\begin{equation}\eqlabel{morBIP}
(\sigma \times \sigma) \circ \tau_X \circ R_g = \tau_{X'} \circ R_{g'} \circ (\sigma \times \sigma).
\end{equation}
Since $(\sigma \times \sigma) \circ \tau_X = \tau_{X'} \circ (\sigma \times \sigma)$, condition
\equref{morBIP} is equivalent to $(\sigma \times \sigma) \circ R_g = R_{g'} \circ (\sigma \times \sigma)$,
i.e. to $\sigma$ being a morphism between the bi-magmas \equref{bimagRg} of $R_g$ and $R_{g'}$.
In particular $\BIP$ can be equally regarded as a full subcategory of the category $\BR$ of
solutions of the braid equation, via $(X, \, R_g)$; and, since every bipartite solution is of
Lyubashenko type (\thref{Rg} (2)), $\BIP$ is a full subcategory of $\LYB$:
\begin{equation*}
\BIP \subseteq \LYB \subseteq \YB.
\end{equation*}
\end{definition}

\begin{definition}\delabel{BIPdat}
A \emph{morphism of bipartite data} $(\alpha, \, \beta) \colon (Y, \, Z, \, g) \to (Y', \, Z', \, g')$
is a pair of maps $\alpha \colon Y \to Y'$ and $\beta \colon Z \to Z'$ such that
\begin{equation}\eqlabel{mordat}
\beta \circ g_y = g'_{\alpha(y)} \circ \beta, \qquad \text{for all} \,\, y \in Y.
\end{equation}
With componentwise composition, bipartite data form a category, which we denote by
$\mathcal{D}at_{\BIP}$.
\end{definition}

\begin{proposition}\prlabel{datfunctor}
The assignment
\begin{equation*}
\Phi \colon \mathcal{D}at_{\BIP} \to \BIP, \qquad
\Phi (Y, \, Z, \, g) := (Y \sqcup Z, \, \tau_X \circ R_g), \qquad
\Phi (\alpha, \, \beta) := \alpha \sqcup \beta
\end{equation*}
is a faithful, not full functor which is surjective on isomorphism classes of objects.
\end{proposition}

\begin{proof}
Let $(\alpha, \, \beta)$ satisfy \equref{mordat} and put $\sigma := \alpha \sqcup \beta$. Since
$\sigma(Y) \subseteq Y'$ and $\sigma(Z) \subseteq Z'$, the six identities of \equref{bimagRg} which
do not involve $g$ are clearly preserved. For $y\in Y$ and $z\in Z$ we have
$\sigma (y \cdot z) = \beta (g_y (z)) = g'_{\alpha(y)} (\beta (z)) = \sigma(y) \cdot \sigma(z)$.
Applying \equref{mordat} to $g_y^{-1}(z)$ gives $\beta(z) = g'_{\alpha(y)} \bigl( \beta (g_y^{-1}(z)) \bigr)$,
i.e. $\sigma (z \ast y) = \beta (g_y^{-1} (z)) = g'^{-1}_{\alpha(y)} (\beta(z)) = \sigma(z) \ast \sigma(y)$.
Thus $\sigma$ is a morphism in $\BIP$. Functoriality is obvious, and $\Phi$ is faithful because
$\alpha$ and $\beta$ are the restrictions of $\alpha \sqcup \beta$ to $Y$ and to $Z$. Surjectivity on
isomorphism classes holds by \deref{BIPcat}.

It remains to show that the functor $\Phi \colon \mathcal{D}at_{\BIP} \to \BIP$ is not full. To this end, it is enough to show that whenever $Y \neq \emptyset \neq Z$,
every constant map $X \to X'$ is a morphism in $\BIP$ which is not of the form $\alpha \sqcup \beta$. Indeed, every bipartite solution is diagonal, $\tau_X R_g (x, \, x) = (x, \, x)$ for all $x \in X$
(both entries of $(x, \, x)$ lie in the same part, so $R_g(x, \, x) = (x, \, x)$). Hence, for any
$c \in X'$, the constant map $\sigma \equiv c$ satisfies \equref{morBIP}: both sides send every
pair to $(c, \, c)$. If $Y \neq \emptyset \neq Z$, a map of the form $\alpha \sqcup \beta$ sends $Y$
into $Y'$ and $Z$ into $Z'$, which are disjoint, so it cannot be constant.
\end{proof}

\begin{remark}\relabel{datvsbip}
Even on isomorphisms the two notions differ, but only in the split case. By \thref{bipiso}(3),
for \emph{non-split} data every isomorphism in $\BIP$ is of the form $\alpha \sqcup \beta$ with
$(\alpha, \, \beta)$ an isomorphism of data. For a split datum with $|Y| = |Z| \geq 1$, a bijection
interchanging $Y$ and $Z$ is an automorphism in $\BIP$ (\thref{autbip}(1)), but not the image of any
morphism of data. Likewise the data $(Y, \, Z, \, {\rm Id})$ and $(Z, \, Y, \, {\rm Id})$ are not
isomorphic in $\mathcal{D}at_{\BIP}$ when $|Y| \neq |Z|$, although they give the same solution.
\end{remark}

\textbf{Convention.} Throughout the rest of the paper, \emph{isomorphism}, \emph{automorphism group} and
\emph{isomorphism class} of bipartite solutions are always understood in the category $\BIP$ of
\deref{BIPcat}, i.e. in the category $\YB$ of all YB-solutions. ${\rm Aut} (X, \, R_g)$ denotes the automorphism group of the braid solution
$(X, \, R_g)$, which coincides, by \equref{auttau}, with that of $(X, \, \tau_X \circ R_g) \in \BIP$. In particular, the isomorphism theorem (\thref{bipiso}), the automorphism groups
(\thref{autbip}) and the count $\mathrm{bip}(n)$ (\coref{countbip}) are statements about
$\BIP$, \emph{not} about the category $\mathcal{D}at_{\BIP}$ of bipartite data; the latter would
give different answers in the split case, see \reref{countdat}.

We now answer the question whether $\BIP$ is equivalent to a variety of algebras. We first recall how products are computed in the category $\YB$. Let $(X, \, r)$ and $(X', \, r')$ be two YB-solutions. On the cartesian product $X \times X'$ we consider the map $r \times r'$ defined, for all $x$, $y\in X$ and $x'$, $y' \in X'$, by
\begin{equation}\eqlabel{prodYB}
(r \times r') \bigl( (x, \, x'), \, (y, \, y') \bigr) := \bigl( (u, \, u'), \, (v, \, v') \bigr), \qquad \text{where} \quad r(x, \, y) = (u, \, v), \,\,\, r'(x', \, y') = (u', \, v').
\end{equation}
In terms of the associated bi-magmas, both operations of $X \times X'$ are defined componentwise. Thus the direct product $(X \times X', \, r \times r')$, together with $\pi_1$ and $\pi_2$, is the product of $(X, \, r)$ and $(X', \, r')$ in $\YB$. If $(W, \, s)$ is a YB-solution and
$f \colon (W, \, s) \to (X, \, r)$, $f' \colon (W, \, s) \to (X', \, r')$ are morphisms in $\YB$, then it is straightforward to see that the map
\begin{equation*}
(f, \, f') \colon W \to X \times X', \qquad (f, \, f')(w) := \bigl( f(w), \, f'(w) \bigr)
\end{equation*}
is the unique morphism of YB-solutions $\varphi: (W, \, s) \to (X \times X', \, r \times r')$ such that $\pi_1 \circ \varphi = f$ and $\pi_2 \circ \varphi = f'$.

\begin{proposition}\prlabel{BIPnotvariety}
Let $E := \{y, \, z\}$ and let $(E, \, {\rm Id}_{E \times E})$ be the bipartite YB-solution associated to the datum $Y := \{y\}$, $Z := \{z\}$ (in this case $R_g = \tau_E$, hence $\tau_E \circ R_g = {\rm Id}_{E \times E}$). Then:
\begin{enumerate}[(1)]
\item The product of $(E, \, {\rm Id}_{E \times E})$ with itself in $\YB$ does not belong to $\BIP$; in particular, $\BIP$ is not closed under products in $\YB$.
\item The product of $(E, \, {\rm Id}_{E \times E})$ with itself does not exist in the category $\BIP$.
\end{enumerate}
Consequently, the category $\BIP$ is not equivalent to any variety of algebras.
\end{proposition}

\begin{proof}
(1) By \equref{prodYB}, the product of $(E, \, {\rm Id}_{E \times E})$ with itself in $\YB$ is $(Q, \, {\rm Id}_{Q \times Q})$, where $Q := E \times E$ is a set with four elements. Assume that $(Q, \, {\rm Id}_{Q \times Q})$ belongs to $\BIP$, i.e. there exist a bipartite datum $(Y', \, Z', \, g')$ and an isomorphism of YB-solutions $\sigma \colon (Q, \, {\rm Id}_{Q \times Q}) \to (X', \, \tau_{X'} \circ R_{g'})$, where $X' = Y' \sqcup Z'$. Then $|X'| = 4$ and $\tau_{X'} \circ R_{g'} = (\sigma \times \sigma) \circ {\rm Id}_{Q \times Q} \circ (\sigma \times \sigma)^{-1} = {\rm Id}_{X' \times X'}$, i.e. $R_{g'} = \tau_{X'}$. Since $|Y'| + |Z'| = 4$, one of the two parts $Y'$, $Z'$ contains two distinct elements $x \neq x'$, and then \equref{Rg} gives $R_{g'} (x, \, x') = (x, \, x') \neq (x', \, x) = \tau_{X'} (x, \, x')$, a contradiction. Hence $(Q, \, {\rm Id}_{Q \times Q}) \notin \BIP$.

(2) Assume that the product of $(E, \, {\rm Id}_{E \times E})$ with itself exists in $\BIP$: that is, there exist an object $(P, \, s) \in \BIP$ and two morphisms $p_1$, $p_2 \colon (P, \, s) \to (E, \, {\rm Id}_{E \times E})$ such that for any object $(W, \, t) \in \BIP$ and any two morphisms $f_1$, $f_2 \colon (W, \, t) \to (E, \, {\rm Id}_{E \times E})$ there exists a unique morphism $\psi \colon (W, \, t) \to (P, \, s)$ with $p_1 \circ \psi = f_1$ and $p_2 \circ \psi = f_2$. We shall prove that $(P, \, s) \cong (Q, \, {\rm Id}_{Q \times Q})$, which contradicts (1).

Let $T := \{\star\}$ be the one-element set; $(T, \, {\rm Id}_{T \times T})$ is the bipartite solution associated to the datum $(\{\star\}, \, \emptyset, \, \emptyset)$, so it is an object of $\BIP$. Every object $(W, \, t)$ of $\BIP$ is of diagonal type by \thref{Rg} (2), since being of diagonal type is preserved by isomorphisms of solutions. Hence, for any $w\in W$, the map $\star \mapsto w$ is a morphism $(T, \, {\rm Id}_{T \times T}) \to (W, \, t)$ in $\BIP$, and all maps $T \to W$ are of this form.

Consider now the map $\varphi := (p_1, \, p_2) \colon P \to Q = E \times E$, $\varphi (w) = \bigl(p_1(w), \, p_2(w)\bigr)$, for all $w \in P$. We show that $\varphi$ is bijective. To this end, let $(a, \, b) \in Q$ be an arbitrary element, i.e. $a$, $b \in E = \{y, \, z\}$. The maps $f_1$, $f_2 \colon T \to E$ given by $f_1 (\star) := a$ and $f_2 (\star) := b$ are morphisms in $\BIP$ from $(T, \, {\rm Id}_{T \times T})$ to $(E, \, {\rm Id}_{E \times E})$. By the universal property of $(P, \, s)$, applied for $(W, \, t) := (T, \, {\rm Id}_{T \times T})$, there exists a unique morphism $\psi \colon T \to P$ such that $p_1 \circ \psi = f_1$ and $p_2 \circ \psi = f_2$. Since the morphisms $T \to P$ are exactly the maps $\star \mapsto w$, with $w\in P$, this means that there exists a unique $w \in P$ such that $p_1 (w) = a$ and $p_2 (w) = b$, i.e. $\varphi (w) = (a, \, b)$. Hence $\varphi$ is bijective.

On the other hand, since $p_1$ and $p_2$ are morphisms of YB-solutions, the universal property of the product $(Q, \, {\rm Id}_{Q \times Q})$ in $\YB$ shows that $\varphi = (p_1, \, p_2)$ is a morphism in $\YB$ from $(P, \, s)$ to $(Q, \, {\rm Id}_{Q \times Q})$. Being bijective, $\varphi$ is an isomorphism of YB-solutions. Thus $(Q, \, {\rm Id}_{Q \times Q}) \cong (P, \, s)$ belongs to $\BIP$, which contradicts (1). Therefore the product of $(E, \, {\rm Id}_{E \times E})$ with itself does not exist in $\BIP$.

Finally, any variety of algebras $\mathcal{V}$ has all products: if $A$, $B \in \mathcal{V}$, then the cartesian product $A \times B$ with componentwise operations satisfies all the identities which hold in $A$ and in $B$, hence it belongs to $\mathcal{V}$, and together with the canonical projections it is the product of $A$ and $B$ in $\mathcal{V}$. In fact, any variety of algebras is a complete and cocomplete category \cite[Proposition 3.4.1, Theorem 3.4.5]{Borceux2}. Since equivalences of categories preserve products, it follows from (2) that $\BIP$ is not equivalent to any variety of algebras.
\end{proof}

\begin{remark}
By \prref{BIPnotvariety} and \thref{lybnotvar} below, neither $\BIP$ nor $\LYB$ is equivalent to a variety of algebras, while
$\BLS$ and $\FS$ are (\cite[Theorem 2.9]{acm1}). The two obstructions are of different natures, although $\BIP$ is a full subcategory of $\LYB$. The category $\LYB$ is closed under products and sub-solutions in $\YB$, and the obstruction is the existence of a non-effective equivalence relation (\thref{lybnotvar}). On the other hand, $\BIP$ is not closed under products in $\YB$ and, even more, the product of two bipartite solutions need not exist in $\BIP$.
\end{remark}

We can recognize the bipartite solutions among all maps
$R \colon X \times X \to X \times X$ (in particular, among all solutions of the braid equation)
directly in terms of $R$. For a map $R$ written as in \equref{notatieR}, let
\begin{equation}\eqlabel{fixR}
{\rm Fix} (R) := \{ (x, \, x') \in X \times X \, \mid \, R (x, \, x') = (x, \, x') \}
= \{ (x, \, x') \, \mid \, x \cdot x' = x \,\,\, \text{and} \,\,\, x \ast x' = x' \},
\end{equation}
regarded as a binary relation on $X$. Clearly, if $\sigma \colon (X, \, R) \to (X', \, R')$ is an
isomorphism, then $(x, \, x') \in {\rm Fix} (R)$ if and only if
$(\sigma(x), \, \sigma(x')) \in {\rm Fix} (R')$.

\begin{theorem}\thlabel{recognition}
Let $X$ be a non-empty set and $R \colon X \times X \to X \times X$ an arbitrary map. Then $(X, \, R)$ is
isomorphic to a bipartite solution $(Y \sqcup Z, \, R_g)$ if and only if the following three
conditions hold:
\begin{enumerate}[(a)]
\item $R^2 = {\rm Id}_{X \times X}$;
\item ${\rm Fix} (R)$ is an equivalence relation on $X$ with at most two equivalence classes;
\item if ${\rm Fix} (R)$ has two classes, then one of them, call it $Y$, satisfies
\begin{equation}\eqlabel{passing}
y \ast x = y \quad \text{and} \quad x \cdot y = y, \qquad \text{for all} \,\, y \in Y \,\, \text{and}
\,\, x \in X \setminus Y,
\end{equation}
i.e. $R (y, \, x) \in X \times \{y\}$ and $R (x, \, y) \in \{y\} \times X$.
\end{enumerate}
In this case $(X, \, R) = (Y \sqcup Z, \, R_g)$ for $Z := X \setminus Y$ and the family
$g = (g_y)_{y \in Y}$ of permutations of $Z$ given by $g_y (z) := y \cdot z$ (with $Y := X$ and
$Z := \emptyset$ if ${\rm Fix} (R)$ has a single class, in which case $R = {\rm Id}_{X \times X}$).
Moreover, if $R$ is a solution of the braid equation, then condition (a) can be replaced by the
bijectivity of $R$.
\end{theorem}

\begin{proof}
Conditions (a)--(c) are invariant under isomorphism, so for the necessity it suffices to check
them for $R_g$ itself. Condition (a) is \thref{Rg} (1). If $Y \neq \emptyset \neq Z$, then by
\equref{Rg} we have $R_g (x, \, x') = (x, \, x')$ when $x$ and $x'$ lie in the same part, while
$R_g (y, \, z) = (g_y (z), \, y) \neq (y, \, z)$ and $R_g (z, \, y) = (y, \, g_y^{-1}(z)) \neq (z, \, y)$;
hence ${\rm Fix} (R_g) = (Y \times Y) \cup (Z \times Z)$, which is (b), and \equref{passing} holds
for the class $Y$ by \equref{bimagRg}. If $Y$ or $Z$ is empty, then $R_g = {\rm Id}$ and
${\rm Fix} (R_g) = X \times X$ has a single class.

Conversely, assume (a)--(c). If ${\rm Fix} (R)$ has a single class, then ${\rm Fix} (R) = X \times X$,
i.e. $R = {\rm Id}_{X \times X}$, which is the bipartite solution of the datum $(X, \, \emptyset, \, \emptyset)$.
Let now $Y$ be a class as in (c) and $Z := X \setminus Y$ the other class. By (b), $R$ is the
identity on $Y \times Y$ and on $Z \times Z$. Let $y \in Y$ and $z \in Z$. By \equref{passing},
$R (y, \, z) = (y \cdot z, \, y)$ and $R (z, \, y) = (y, \, z \ast y)$. We claim that $y \cdot z \in Z$.
Indeed, otherwise $(y \cdot z, \, y) \in Y \times Y \subseteq {\rm Fix} (R)$, and then, by (a),
$(y, \, z) = R^2 (y, \, z) = R (y \cdot z, \, y) = (y \cdot z, \, y)$, which is impossible since
$z \neq y$ lie in different classes. In the same way $z \ast y \in Z$. Thus we obtain two maps
$g_y, \, h_y \colon Z \to Z$, $g_y (z) := y \cdot z$ and $h_y (z) := z \ast y$, and (a) gives
\begin{equation*}
(y, \, z) = R^2 (y, \, z) = R \bigl( g_y (z), \, y \bigr) = \bigl( y, \, h_y g_y (z) \bigr), \qquad
(z, \, y) = R^2 (z, \, y) = R \bigl( y, \, h_y (z) \bigr) = \bigl( g_y h_y (z), \, y \bigr),
\end{equation*}
i.e. $g_y \in \Sym{Z}$ and $h_y = g_y^{-1}$. Comparing with \equref{Rg}, we get $R = R_g$.

Finally, let $R$ be a bijective solution of the braid equation satisfying (b) and (c), with $Y$
and $Z$ as above (the case of a single class being trivial). Since $R$ is the identity on
$(Y \times Y) \cup (Z \times Z)$ and is bijective, it maps the complement $(Y \times Z) \cup (Z \times Y)$
onto itself. Hence, by \equref{passing}, $R (y, \, z) = (g_y (z), \, y)$ and
$R (z, \, y) = (y, \, h_y (z))$ for some maps $g_y, \, h_y \colon Z \to Z$. Evaluating both sides of
the braid equation on the triple $(z, \, y, \, z')$, with $z$, $z' \in Z$ and $y \in Y$, we obtain
\begin{equation*}
(z, \, y, \, z') \overset{R^{12}}{\longmapsto} \bigl( y, \, h_y (z), \, z' \bigr) \overset{R^{23}}{\longmapsto}
\bigl( y, \, h_y (z), \, z' \bigr) \overset{R^{12}}{\longmapsto} \bigl( g_y h_y (z), \, y, \, z' \bigr),
\end{equation*}
\begin{equation*}
(z, \, y, \, z') \overset{R^{23}}{\longmapsto} \bigl( z, \, g_y (z'), \, y \bigr) \overset{R^{12}}{\longmapsto}
\bigl( z, \, g_y (z'), \, y \bigr) \overset{R^{23}}{\longmapsto} \bigl( z, \, y, \, h_y g_y (z') \bigr).
\end{equation*}
Hence $g_y h_y = {\rm Id}_Z = h_y g_y$ for all $y \in Y$, which gives (a), and we conclude by the
first part.
\end{proof}

\begin{remark}\relabel{recognition}
%(1) In the language of the YB-solution $r := \tau_X \circ R$, the relation \equref{fixR} becomes
%${\rm Fix} (R) = \{ (x, \, x') \mid r (x, \, x') = (x', \, x) \}$, while \equref{passing} says that
%an element of $Y$ is never moved by $r$ when it meets an element of the other class:
%$r (y, \, x) \in \{y\} \times X$ and $r (x, \, y) \in X \times \{y\}$ for all $x \notin Y$.

None of the three conditions of \thref{recognition} can be dropped, even for solutions of the
braid equation. Indeed, first we observe that the flip $\tau_X$ with $|X| \geq 3$ is an involutive solution of the braid equation for which
${\rm Fix} (\tau_X)$ is the equality relation, with $|X| \geq 3$ classes. Thus (a) holds and (c) holds trivially, while (b)
fails. Secondly, on the set $X := \{0, \, 1\}$, the map $R$ with $R (0, \, 0) = R (0, \, 1) = R (1, \, 0) = (0, \, 0)$ and
$R (1, \, 1) = (1, \, 1)$ is a solution of the braid equation satisfying (b) and (c) (with
$Y = \{0\}$), but not (a), since $R^2 (0, \, 1) = (0, \, 0) \neq (0, \, 1)$. Finally, let $Y$ and $Z$ be disjoint sets with $|Y| = |Z| = m \geq 2$,
$X = Y \sqcup Z$, and let $\sigma$ be an involution of $X$ with $\sigma (Y) = Z$. Then the map
\begin{equation*}
R (x, \, x') := \begin{cases} (x, \, x'), & \text{if} \,\, x, \, x' \,\, \text{lie in the same part},\\
\bigl( \sigma (x), \, \sigma (x') \bigr), & \text{otherwise}
\end{cases}
\end{equation*}
is an involutive solution of the braid equation with ${\rm Fix} (R) = (Y \times Y) \cup (Z \times Z)$,
so (a) and (b) hold; however, for $y \in Y$ and $z \in Z$ with $z \neq \sigma (y)$ we have
$R (y, \, z) = (\sigma (y), \, \sigma (z))$ with $\sigma (z) \neq y$, so neither $Y$ nor (symmetrically)
$Z$ satisfies \equref{passing}. Thus (c) fails and $R$ is not bipartite. (For $m = 1$ this map is
just the flip on two points, which is bipartite.)
\end{remark}

The next lemma shows that the partition $X = Y \sqcup Z$ is \emph{intrinsic}, i.e. recoverable from
the solution $R_g$ alone; this is the key to both the isomorphism problem and the computation of the
automorphism group below.

\begin{lemma}\lelabel{intrinsic}
Let $(Y, \, Z, \, g)$ be a bipartite datum, $X = Y \sqcup Z$, and let $(X, \, \cdot, \, \ast)$ be
the bi-magma \equref{bimagRg} of $R_g$.
\begin{enumerate}[(1)]
\item For all $x$, $x'\in X$ we have $x \cdot x' = x$ if and only if $x$ and $x'$ lie in the same part.
Hence, if $Y \neq \emptyset \neq Z$, the relation $x \sim x' :\Longleftrightarrow x \cdot x' = x$ is an
equivalence relation on $X$ whose classes are exactly $Y$ and $Z$, and therefore the
\emph{unordered} partition $\{Y, \, Z\}$ is determined by $R_g$.
\item Consider the property
\begin{equation*}
(P) \qquad x \ast x' = x \,\,\,\, \text{for all} \,\, x' \in X \,\, \text{with} \,\, x \cdot x' \neq x
\end{equation*}
of an element $x\in X$. Then every $y \in Y$ has property $(P)$, while $z\in Z$ has property $(P)$
if and only if $g_y (z) = z$ for all $y\in Y$. In particular, the datum is split if and only if
every element of $X$ has property $(P)$.
\item If $(Y, \, Z, \, g)$ and $(Y', \, Z', \, g')$ are two non-split bipartite data and
$\varphi \colon (X, \, R_g) \to (X', \, R_{g'})$ is an isomorphism, then $\varphi (Y) = Y'$ and
$\varphi (Z) = Z'$.
\end{enumerate}
\end{lemma}

\begin{proof}
(1) If $x$ and $x'$ lie in the same part, then $x \cdot x' = x$ by the first and third equalities in
\equref{bimagRg}. If they lie in different parts, then either $x = y\in Y$ and $x' = z \in Z$, in
which case $x \cdot x' = g_y (z) \in Z$ and hence $x \cdot x' \neq y = x$; or $x = z\in Z$ and
$x' = y\in Y$, in which case $x \cdot x' = y \neq z = x$. The remaining assertions are clear.

(2) Let $y \in Y$. By (1), the elements $x'$ with $y \cdot x' \neq y$ are exactly the $x' \in Z$, and
$y \ast z = y$ for all $z\in Z$ by \equref{bimagRg}; so $y$ has $(P)$. Let now $z \in Z$. By (1),
the elements $x'$ with $z \cdot x' \neq z$ are exactly the $x' \in Y$, and $z \ast y = g_y^{-1} (z)$;
so $z$ has $(P)$ if and only if $g_y^{-1} (z) = z$, i.e. $g_y (z) = z$, for all $y\in Y$. The last
assertion follows, since the datum is split precisely when every $z\in Z$ is fixed by every $g_y$.

(3) Since $\varphi$ is an isomorphism of bi-magmas and a bijection, we have $x\cdot x' = x$ if and
only if $\varphi (x) \cdot \varphi(x') = \varphi(x)$; hence, by (1), $\varphi$ carries the unordered
partition $\{Y, \, Z\}$ onto $\{Y', \, Z'\}$, i.e. either $\varphi(Y) = Y'$, $\varphi(Z) = Z'$, or
$\varphi(Y) = Z'$, $\varphi(Z) = Y'$. Note that $Y \neq \emptyset \neq Z$ and
$Y' \neq \emptyset \neq Z'$, both data being non-split. For the same reason, $x$ has property
$(P)$ if and only if $\varphi(x)$ has property $(P)$. By (2), every element of $Y$ has $(P)$,
while some $z_0 \in Z$ does not; likewise some $z_0' \in Z'$ does not have $(P)$. If we had
$\varphi (Z) = Y'$, then $\varphi (z_0) \in Y'$ would have property $(P)$ and hence so would
$z_0$, a contradiction. Thus $\varphi(Y) = Y'$ and $\varphi(Z) = Z'$.
\end{proof}

We recall that a \emph{multiset} is a set in which elements may be repeated, e.g. $\{x, \, x, \, x, \, y, \, \dots \}$.
More rigorously, a multiset is an ordered pair $(S, \, m)$, where $S$ is a usual set and $m$ is a function
which assigns to each $s \in S$ a non-zero cardinal number $m(s)$, the \emph{multiplicity} of $s$ in the multiset; if all the multiplicities are finite, then $m \colon S \to \mathbb{N}^{*}$.

For a bipartite datum $(Y, \, Z, \, g)$ we define the sets
\begin{equation}\eqlabel{fibre}
G_g := \{g_y \, | \, y\in Y\} \subseteq \Sym{Z}, \qquad Y_h := \{y \in Y \, | \, g_y = h\} \,\,\,\,\, \text{for} \,\, h \in G_g,
\end{equation}
so that $(Y_h)_{h \in G_g}$ is the partition of $Y$ into the non-empty fibres of $g$. We denote by
$M_g$ the multiset of values of $g$, i.e. the multiset $(G_g, \, m_g)$, where $m_g (h) := |Y_h|$, for all $h \in G_g$. For $\beta \in \Sym{Z}$ we let
$$
c_\beta \colon \Sym{Z} \to \Sym{Z}, \qquad c_\beta (h) := \beta \, h \, \beta^{-1}
$$
be the conjugation by $\beta$, and we denote by $c_\beta (M_g)$ the image of the multiset $M_g$ through $c_\beta$, i.e. the
multiset with underlying set $\beta \, G_g \, \beta^{-1}$ in which $\beta \, h \, \beta^{-1}$ has multiplicity $|Y_h|$, for all $h \in G_g$. If $(Y', \, Z, \, g')$ is another bipartite datum with the same set $Z$ and $Y'_{h'} := \{y' \in Y' \, | \, g'_{y'} = h'\}$, for $h' \in G_{g'}$, are the fibres of $g'$, then $c_\beta (M_g) = M_{g'}$ means that
$\beta \, G_g \, \beta^{-1} = G_{g'}$ and $|Y_h| = |Y'_{\beta h \beta^{-1}}|$, for all $h\in G_g$.

We can now settle the isomorphism problem for bipartite solutions. Recall from \equref{auttau}
that $R_g$ and $\tau_X \circ R_g$ have the same morphisms, so the statement below holds verbatim
for the YB-solutions $\tau_X \circ R_g$. Below, $\cong$ denotes isomorphism of solutions in $\BIP$; equivalently, of the braid solutions
$R_g$ and $R_{g'}$.

\begin{theorem}\thlabel{bipiso}
Let $(Y, \, Z, \, g)$ and $(Y', \, Z', \, g')$ be two bipartite data, $X = Y \sqcup Z$ and
$X' = Y' \sqcup Z'$.
\begin{enumerate}[(1)]
\item If both data are split, then $(X, \, R_g) \cong (X', \, R_{g'})$ if and only if
$\{|Y|, \, |Z|\} = \{|Y'|, \, |Z'|\}$ as unordered pairs of cardinals.
\item If one datum is split and the other is not, then $(X, \, R_g) \not\cong (X', \, R_{g'})$.
\item If both data are non-split, then $(X, \, R_g) \cong (X', \, R_{g'})$ if and only if there
exist bijections $\alpha \colon Y \to Y'$ and $\beta \colon Z \to Z'$ such that
\begin{equation}\eqlabel{isocrit}
g'_{\alpha (y)} = \beta \, g_y \, \beta^{-1}, \qquad \text{for all} \,\, y \in Y.
\end{equation}
Equivalently: $|Z| = |Z'|$ and, after identifying $Z$ with $Z'$ by means of an arbitrary
bijection, there exists $\beta \in \Sym{Z}$ with $c_\beta (M_g) = M_{g'}$, i.e. the
multisets of values of $g$ and $g'$ are conjugate in $\Sym{Z}$.
\end{enumerate}
\end{theorem}

\begin{proof}
(1) For a split datum, $R_g = \tau_{Y, Z}$ depends only on the partition. If
$\{|Y|, \, |Z|\} = \{|Y'|, \, |Z'|\}$, choose a bijection $\varphi \colon X \to X'$ carrying $Y$,
$Z$ onto $Y'$, $Z'$ or onto $Z'$, $Y'$. In the first case $\varphi$ is clearly an isomorphism. In
the second case it is one as well, because the bi-magma \equref{bimagRg} of a split datum is
\emph{symmetric} in the two parts: it reads $y\cdot y' = y$, $y \ast y' = y'$,
$z \cdot z' = z$, $z \ast z' = z'$, $y \cdot z = z$, $y \ast z = y$, $z\cdot y = y$,
$z \ast y = z$, and interchanging the roles of $Y$ and $Z$ maps each of these eight identities to
another one of the list. Conversely, if $\varphi$ is an isomorphism then, by
\leref{intrinsic}(1), it carries $\{Y, \, Z\}$ onto $\{Y', \, Z'\}$ (if one of the parts is empty,
then $R_g = {\rm Id}_{X\times X}$ and $|X| = |X'|$ forces
$\{|Y|, \, |Z|\} = \{|X|, \, 0\} = \{|Y'|, \, |Z'|\}$, since a split datum with a part equal to
$X$ also gives ${\rm Id}$).

(2) Property $(P)$ is preserved by isomorphisms, as noted in the proof of \leref{intrinsic}(3).
By \leref{intrinsic}(2) every element of $X$ has property $(P)$ in the split case, while in the
non-split case some element does not.

(3) Assume \equref{isocrit} and let $\varphi := \alpha \sqcup \beta$, the bijection
$X \to X'$ equal to $\alpha$ on $Y$ and to $\beta$ on $Z$. Since $\varphi$ maps $Y$ to $Y'$ and
$Z$ to $Z'$, it preserves the first four equalities in \equref{bimagRg}, as well as
$y \ast z = y$ and $z \cdot y = y$. For the remaining two,
$\varphi (y \cdot z) = \beta (g_y (z)) = g'_{\alpha(y)} (\beta(z)) = \varphi (y) \cdot \varphi (z)$
and
$\varphi (z \ast y) = \beta (g_y^{-1} (z)) = g'^{-1}_{\alpha(y)} (\beta(z)) = \varphi (z) \ast \varphi (y)$,
since $\beta g_y^{-1} \beta^{-1} = (\beta g_y \beta^{-1})^{-1} = g'^{-1}_{\alpha(y)}$. Hence
$\varphi$ is an isomorphism.

Conversely, let $\varphi \colon (X, \, R_g) \to (X', \, R_{g'})$ be an isomorphism. By
\leref{intrinsic}(3) we have $\varphi(Y) = Y'$ and $\varphi(Z) = Z'$, so $\varphi = \alpha \sqcup \beta$
with $\alpha := \varphi|_Y \colon Y \to Y'$ and $\beta := \varphi|_Z \colon Z \to Z'$ bijections. For
$y\in Y$ and $z\in Z$ we obtain
$\beta (g_y (z)) = \varphi (y \cdot z) = \varphi(y) \cdot \varphi(z) = g'_{\alpha(y)} (\beta(z))$,
i.e. $\beta g_y = g'_{\alpha(y)} \beta$, which is \equref{isocrit}.

Finally we prove the equivalence with the multiset formulation. Identify $Z$ with $Z'$ along a
bijection, so that $g$ and $g'$ take values in the same group $\Sym{Z}$. If $(\alpha, \beta)$
satisfies \equref{isocrit}, then for every $h \in G_g$ and every $y \in Y_h$ we get
$g'_{\alpha(y)} = c_\beta (h)$; thus $\alpha (Y_h) \subseteq Y'_{c_\beta (h)}$, and since $\alpha$
is a bijection and the fibres partition $Y$ and $Y'$, we get $\alpha (Y_h) = Y'_{c_\beta (h)}$ for
all $h$. Hence $c_\beta (G_g) = G_{g'}$ and $|Y_h| = |Y'_{c_\beta (h)}|$, i.e.
$c_\beta (M_g) = M_{g'}$. Conversely, if $c_\beta (M_g) = M_{g'}$ for some
$\beta \in \Sym{Z}$, choose for every $h \in G_g$ a bijection
$\alpha_h \colon Y_h \to Y'_{c_\beta (h)}$ --- possible since the two sets have the same
cardinality --- and let $\alpha$ be the disjoint union of the $\alpha_h$; then
$(\alpha, \, \beta)$ satisfies \equref{isocrit}.
\end{proof}

\begin{corollary}\colabel{countbip}
For a finite set $X$ with $|X| = n$, the number $\mathrm{bip} (n)$ of isomorphism classes of bipartite solutions on $X$ is given by:
\begin{equation}\eqlabel{countformula}
\mathrm{bip} (n) = \Bigl\lfloor \frac{n}{2} \Bigr\rfloor + 1 + \sum_{m=2}^{n} \bigl( o(m, \, n-m) - 1 \bigr),
\end{equation}
where $o(m, \, k)$ denotes the number of orbits of the action of $\Sym{m}$ by conjugation on the set of all multisets of cardinality $k$ of elements of $\Sym{m}$. Moreover,
\begin{equation}\eqlabel{omk}
o(m, \, k) = \frac{1}{m!} \sum_{\beta \in \Sym{m}} N_\beta (k),
\end{equation}
where, if $\mathcal{O}_1, \ldots, \mathcal{O}_r$ are the orbits of the cyclic group generated by the conjugation $c_\beta \colon \Sym{m} \to \Sym{m}$, $c_\beta (h) = \beta \, h \, \beta^{-1}$, acting on the set $\Sym{m}$, then
\begin{equation*}
N_\beta (k) := \Bigl| \Bigl\{ (n_1, \, \ldots, \, n_r) \in \mathbb{N}^r \, \Bigl| \, |\mathcal{O}_1| \, n_1 + \cdots + |\mathcal{O}_r| \, n_r = k \Bigr\} \Bigr|
\end{equation*}
is the number of solutions in non-negative integers of the equation $|\mathcal{O}_1| \, n_1 + \cdots + |\mathcal{O}_r| \, n_r = k$. In particular, the number of isomorphism classes of non-split bipartite solutions with $|Z| = 2$ is $n - 2$, and the number of those with $|Y| = 1$ is $p(n-1) - 1$, where $p$ is the Euler partition function.
\end{corollary}

\begin{proof}
By \thref{bipiso}(1)--(2), the isomorphism classes of split bipartite solutions on $X$ are in bijection with the unordered pairs $\{|Y|, \, |Z|\}$ of non-negative integers with $|Y| + |Z| = n$; there are $\lfloor n/2 \rfloor + 1$ such pairs, and no split solution is isomorphic to a non-split one. By \thref{bipiso}(3), the isomorphism classes of non-split bipartite solutions with $|Z| = m$, and hence $|Y| = n - m$, are in bijection with the orbits of the action of $\Sym{m}$ by conjugation on the multisets of cardinality $n - m$ of elements of $\Sym{m}$, except for the orbit of the multiset $\{{\rm Id}_Z, \ldots, {\rm Id}_Z\}$, which corresponds to the split datum. This gives $o(m, \, n-m) - 1$ classes. If $m \leq 1$, then $\Sym{m}$ is trivial, every datum is split and there is nothing to count; this is why the sum in \equref{countformula} starts at $m = 2$.

It remains to prove \equref{omk}. Let $\mathcal{M}_k$ be the set of all multisets of cardinality $k$ of elements of $\Sym{m}$, i.e. the set of all functions $\mu \colon \Sym{m} \to \mathbb{N}$ with $\sum_{h\in \Sym{m}} \mu (h) = k$, where $\mu (h)$ is the multiplicity of $h$. The group $\Sym{m}$ acts on $\mathcal{M}_k$ by conjugation: $(\beta \cdot \mu) (h) := \mu \bigl(c_\beta^{-1} (h)\bigr) = \mu (\beta^{-1} h \, \beta)$, for all $\beta$, $h \in \Sym{m}$. By the Cauchy--Frobenius lemma \cite[Theorem 3.22]{Rotman}\footnote{If a finite group $G$ acts on a finite set $S$, then the number of orbits is equal to $|G|^{-1} \sum_{g\in G} |\{s \in S \, | \, g\cdot s = s\}|$. The result goes back to Cauchy and was published in its modern form by Frobenius \cite{Frobenius87}; it is commonly called \emph{Burnside's lemma}, although Burnside \cite{Burnside97} stated it crediting Frobenius. See Neumann \cite{Neumann79}, whose title says it all: \emph{A lemma that is not Burnside's}.},
\begin{equation*}
o(m, \, k) = \frac{1}{m!} \sum_{\beta \in \Sym{m}} \bigl| \{ \mu \in \mathcal{M}_k \, | \, \beta \cdot \mu = \mu \} \bigr|.
\end{equation*}
Fix $\beta \in \Sym{m}$. A multiset $\mu$ is fixed by $\beta$ if and only if $\mu (c_\beta (h)) = \mu (h)$ for all $h \in \Sym{m}$, i.e. if and only if $\mu$ is constant on each orbit $\mathcal{O}_i$ of the cyclic group generated by $c_\beta$. Such a $\mu$ is uniquely determined by its constant values $n_i \in \mathbb{N}$ on the orbits $\mathcal{O}_i$, $i = 1, \ldots, r$, and its cardinality is $\sum_{h} \mu (h) = |\mathcal{O}_1| \, n_1 + \cdots + |\mathcal{O}_r| \, n_r$. Hence the number of multisets of cardinality $k$ fixed by $\beta$ is exactly $N_\beta (k)$, and \equref{omk} follows.

Finally, if $m = 2$, then $\Sym{2}$ is abelian, so the conjugation action is trivial and $o(2, \, k)$ is the number of multisets of cardinality $k$ of a two-element set, namely $k + 1$; for $k = n - 2$ we obtain $o(2, \, n-2) - 1 = n - 2$. If $|Y| = 1$, i.e. $k = 1$ and $m = n - 1$, then the multisets of cardinality $1$ are the elements of $\Sym{m}$ and their orbits are the conjugacy classes of $\Sym{m}$; there are $p(m) = p(n-1)$ of them, one of which is $\{{\rm Id}_Z\}$.
\end{proof}

\begin{remark}\relabel{countdat}
The choice of morphisms matters for the count, but only through the split data. If one counts
instead the isomorphism classes of bipartite data with $|Y| + |Z| = n$ in the category
$\mathcal{D}at_{\BIP}$ of \deref{BIPdat}, the non-split classes are the same as in $\BIP$ (an
isomorphism of data $(\alpha, \, \beta)$ is exactly a pair satisfying \equref{isocrit}), whereas
the split classes are now indexed by the \emph{ordered} pairs $(|Y|, \, |Z|)$, since an
isomorphism of data cannot interchange $Y$ and $Z$. Hence the number of isomorphism classes of
bipartite data of size $n$ is
\begin{equation*}
(n + 1) + \sum_{m=2}^{n} \bigl( o(m, \, n-m) - 1 \bigr) = \mathrm{bip} (n) + \Bigl\lceil \frac{n}{2} \Bigr\rceil.
\end{equation*}
Similarly, for a split datum with $|Y| = |Z| = m \geq 1$ the automorphism group in $\mathcal{D}at_{\BIP}$ is
$\Sym{Y} \times \Sym{Z} \cong \Sym{m} \times \Sym{m}$, of index $2$ in the group $(\Sym{Y} \times \Sym{Z}) \rtimes_{\omega} \bZ/2$ of \thref{autbip}(1);
for non-split data the two automorphism groups coincide, by \prref{autRg}.
\end{remark}

Applying \thref{bipiso}(3) with $(Y', \, Z', \, g') = (Y, \, Z, \, g)$ we recover, for non-split
data, the description of the automorphism group. We state it separately, since it will be used
in the proof of \thref{autbip}.

\begin{proposition}\prlabel{autRg}
Let $(Y, \, Z, \, g)$ be a non-split bipartite datum and $X = Y \sqcup Z$. Then the map
\begin{equation*}
H := \Bigl\{ (\alpha, \, \beta) \in \Sym{Y} \times \Sym{Z} \, \bigl| \, \beta \circ g_y \circ \beta^{-1} = g_{\alpha(y)}, \,\, \text{for all} \,\, y\in Y \Bigr\} \to {\rm Aut} (X, \, R_g), \qquad (\alpha, \, \beta) \mapsto \alpha \sqcup \beta,
\end{equation*}
is an isomorphism of groups, where $\alpha \sqcup \beta$ is the bijection of $X = Y \sqcup Z$ which is equal to $\alpha$ on $Y$ and to $\beta$ on $Z$; the left hand side is a subgroup of $\Sym{Y} \times \Sym{Z}$.
\end{proposition}

\begin{proof}
That $H$ is a subgroup of $\Sym{Y} \times \Sym{Z}$ is straightforward: if
$(\alpha, \, \beta)$, $(\alpha', \, \beta') \in H$, then
\begin{equation*}
\beta \beta' g_y (\beta \beta')^{-1} = \beta g_{\alpha'(y)} \beta^{-1} = g_{\alpha \alpha' (y)}
\end{equation*}
and $\beta^{-1} g_y \beta = g_{\alpha^{-1} (y)}$, as follows by applying the defining relation of
$H$ to $\alpha^{-1} (y)$. Also, $(\alpha, \, \beta) \mapsto \alpha \sqcup \beta$ is an injective
morphism of groups from $\Sym{Y} \times \Sym{Z}$ to $\Sym{X}$. It takes values in
${\rm Aut} (X, \, R_g)$ and is surjective onto it by \thref{bipiso}(3), applied with
$(Y', \, Z', \, g') = (Y, \, Z, \, g)$: the two implications proved there say precisely that
$\alpha \sqcup \beta$ is an automorphism if and only if $(\alpha, \, \beta) \in H$, and that every
automorphism is of this form.
\end{proof}

The group $H$ of \prref{autRg} is not an arbitrary subgroup of
$\Sym{Y} \times \Sym{Z}$: it is an extension of the stabilizer of the multiset $M_g$
by the direct product of the symmetric groups of the fibres of $g$, and the extension splits. The
split case, excluded in \prref{autRg}, is also settled below.

\begin{theorem}\thlabel{autbip}
Let $(Y, \, Z, \, g)$ be a bipartite datum, $X = Y \sqcup Z$, and let
${\rm Aut} \, (X, \, R_g)$ be the automorphism group of the braid solution
$(X, \, R_g)$.\footnote{Equivalently, of the YB-solution $(X, \, \tau_X \circ R_g)$, by \equref{auttau}.} Then:
\begin{enumerate}[(1)]
\item If the datum $(Y, \, Z, \, g)$  is split, then
\begin{equation*}
{\rm Aut} (X, \, R_g) = \begin{cases}
\Sym{Y} \times \Sym{Z}, & \text{if} \,\, |Y| \neq |Z|,\\
\bigl( \Sym{Y} \times \Sym{Z}\bigr) \rtimes_{\omega} \bZ/2, & \text{if} \,\, |Y| = |Z| = m \geq 1,
\end{cases}
\end{equation*}
where, for a fixed bijection $\theta \colon Y \to Z$, the generator of $\bZ/2$ acts on $\Sym{Y} \times \Sym{Z}$ by the automorphism $\omega (\alpha, \, \beta) := (\theta^{-1} \beta \, \theta, \, \theta \, \alpha \, \theta^{-1})$. If $Y = \emptyset$ or $Z = \emptyset$, then
$R_g = {\rm Id}_{X \times X}$ and ${\rm Aut} (X, \, R_g) = \Sym{X}$.
\item If the datum $(Y, \, Z, \, g)$  is non-split, let
\begin{align*}
N_g &:= \bigl\{ \beta \in \Sym{Z} \, \bigl| \, c_\beta (M_g) = M_g \bigr\}\\
&\phantom{:}= \bigl\{ \beta \in \Sym{Z} \, \bigl| \, \beta G_g \beta^{-1} = G_g \,\, \text{and} \,\, |Y_{\beta h \beta^{-1}}| = |Y_h| \,\, \text{for all} \,\, h\in G_g \bigr\}
\end{align*}
be the stabilizer of the multiset $M_g$ under conjugation, a subgroup of $\Sym{Z}$. Then the
second projection ${\rm Aut} (X, \, R_g) = H \to \Sym{Z}$, $(\alpha, \, \beta) \mapsto \beta$,
has image $N_g$ and kernel $\prod_{h \in G_g} \Sym{Y_h}$, and the resulting short exact
sequence of groups
\begin{equation}\eqlabel{extaut}
1 \longrightarrow \prod_{h \in G_g} \Sym{Y_h} \longrightarrow {\rm Aut} (X, \, R_g) \longrightarrow N_g \longrightarrow 1
\end{equation}
is split. Consequently
\begin{equation}\eqlabel{autsemi}
{\rm Aut} (X, \, R_g) \cong \Bigl( \prod_{h \in G_g} \Sym{Y_h} \Bigr) \rtimes N_g,
\end{equation}
where $\beta \in N_g$ acts on the direct product by mapping each factor $\Sym{Y_h}$ isomorphically onto $\Sym{Y_{\beta h \beta^{-1}}}$, by conjugation with a bijection $Y_h \to Y_{\beta h \beta^{-1}}$ (see the proof below); in particular, the factors are permuted along the bijection
$h \mapsto \beta h \beta^{-1}$ of $G_g$.
\item If moreover the map $Y \to \Sym{Z}$, $y \mapsto g_y$, is injective, then all fibres
$Y_h$ are singletons, the kernel in \equref{extaut} is trivial and
\begin{equation*}
{\rm Aut} (X, \, R_g) \cong N_g = \bigl\{\beta \in \Sym{Z} \, \bigl| \, \beta \, G_g \, \beta^{-1} = G_g \bigr\} = N_{\Sym{Z}} (G_g),
\end{equation*}
the normalizer in $\Sym{Z}$ of the \emph{set} $G_g = \{g_y \, | \, y\in Y\}$.
\end{enumerate}
\end{theorem}

\begin{proof}
(1) If $Y = \emptyset$ or $Z = \emptyset$ then $R_g = {\rm Id}_{X\times X}$ and every bijection of
$X$ is an automorphism. Assume $Y \neq \emptyset \neq Z$. By \leref{intrinsic}(1), an automorphism
$\varphi$ carries the unordered partition $\{Y, \, Z\}$ onto itself, so either it preserves both
parts, in which case $\varphi \in \Sym{Y} \times \Sym{Z}$, or it interchanges them, in
which case $|Y| = |Z|$. Conversely, as shown in the proof of \thref{bipiso}(1), the bi-magma of a
split datum is symmetric in the two parts, so every bijection of $X$ preserving or interchanging
$Y$ and $Z$ is an automorphism. If $|Y| \neq |Z|$ the second possibility does not occur, which
gives the first formula; here and below we identify $\Sym{Y} \times \Sym{Z}$ with the subgroup of $\Sym{X}$ consisting of all bijections $\alpha \sqcup \beta$ which preserve $Y$ and $Z$,
where $\alpha \sqcup \beta$ is equal to $\alpha$ on $Y$ and to $\beta$ on $Z$.

Assume now that $|Y| = |Z| = m \geq 1$ and let $\theta \colon Y \to Z$ be a fixed bijection. We define $\sigma \in \Sym{X}$ by $\sigma (y) := \theta (y)$ and $\sigma (z) := \theta^{-1} (z)$, for all $y\in Y$ and $z\in Z$. Then $\sigma$ interchanges $Y$ and $Z$, so $\sigma \in {\rm Aut} (X, \, R_g)$, and $\sigma^2 = {\rm Id}_X$. The map
\begin{equation*}
\omega \colon \Sym{Y} \times \Sym{Z} \to \Sym{Y} \times \Sym{Z}, \qquad \omega (\alpha, \, \beta) := (\theta^{-1} \beta \, \theta, \, \theta \, \alpha \, \theta^{-1})
\end{equation*}
is a group automorphism, since $\Sym{Y} \to \Sym{Z}$, $\alpha \mapsto \theta \, \alpha \, \theta^{-1}$, is an isomorphism of groups with inverse $\beta \mapsto \theta^{-1} \beta \, \theta$. Moreover, $\omega^2 = {\rm Id}$, since
\begin{equation*}
\omega^2 (\alpha, \, \beta) = \omega (\theta^{-1} \beta \, \theta, \, \theta \, \alpha \, \theta^{-1}) = (\theta^{-1} \theta \, \alpha \, \theta^{-1} \theta, \, \theta \, \theta^{-1} \beta \, \theta \, \theta^{-1}) = (\alpha, \, \beta).
\end{equation*} Hence $\bZ/2 = \{\overline{0}, \, \overline{1}\}$ acts on $\Sym{Y} \times \Sym{Z}$ via the group morphism $\bZ/2 \to {\rm Aut} \, (\Sym{Y} \times \Sym{Z})$, $\overline{0} \mapsto {\rm Id}$, $\overline{1} \mapsto \omega$, and we can consider the semidirect product $(\Sym{Y} \times \Sym{Z}) \rtimes_{\omega} \bZ/2$: its underlying set is $(\Sym{Y} \times \Sym{Z}) \times \bZ/2$, with the multiplication given by
\begin{equation*}
\bigl( (\alpha, \, \beta), \, a \bigr) \cdot \bigl( (\alpha', \, \beta'), \, b \bigr) := \bigl( (\alpha, \, \beta) \, \omega^{a} (\alpha', \, \beta'), \, a + b \bigr),
\end{equation*}
for all $\alpha$, $\alpha' \in \Sym{Y}$, $\beta$, $\beta' \in \Sym{Z}$ and $a$, $b\in \bZ/2$. The action $\omega$ is precisely the conjugation by $\sigma$: for all $y \in Y$ and $z \in Z$ we have
\begin{equation*}
\sigma \, (\alpha \sqcup \beta) \, \sigma^{-1} (y) = \theta^{-1} \beta \, \theta (y), \qquad \sigma \, (\alpha \sqcup \beta) \, \sigma^{-1} (z) = \theta \, \alpha \, \theta^{-1} (z),
\end{equation*}
i.e. $\sigma \, (\alpha \sqcup \beta) \, \sigma^{-1} = \omega (\alpha, \, \beta)$, the latter being viewed as an element of $\Sym{X}$. We prove that the map
\begin{equation*}
\Phi \colon (\Sym{Y} \times \Sym{Z}) \rtimes_{\omega} \bZ/2 \to {\rm Aut} \, (X, \, R_g), \qquad \Phi \bigl( (\alpha, \, \beta), \, a \bigr) := (\alpha \sqcup \beta) \circ \sigma^{a}
\end{equation*}
is an isomorphism of groups, where $\sigma^{\overline{0}} := {\rm Id}_X$ and $\sigma^{\overline{1}} := \sigma$ (this is well defined since $\sigma^2 = {\rm Id}_X$). Indeed, $\Phi$ is a morphism of groups since, writing $k := \alpha \sqcup \beta$ and $k' := \alpha' \sqcup \beta'$, we have
\begin{equation*}
\begin{split}
\Phi \bigl( (\alpha, \, \beta), \, a \bigr) \, \Phi \bigl( (\alpha', \, \beta'), \, b \bigr) &= k \, \sigma^{a} \, k' \, \sigma^{b} = k \, \bigl(\sigma^{a} k' \sigma^{-a} \bigr) \, \sigma^{a+b} = k \, \omega^{a} (\alpha', \, \beta') \, \sigma^{a+b}\\
&= \Phi \Bigl( \bigl( (\alpha, \, \beta), \, a \bigr) \cdot \bigl( (\alpha', \, \beta'), \, b \bigr) \Bigr).
\end{split}
\end{equation*}
If $k \, \sigma^{a} = {\rm Id}_X$, then $\sigma^{a} = k^{-1}$ preserves $Y$, which forces $a = \overline{0}$ (as $Y \neq \emptyset$ and $\sigma (Y) = Z$), and then $k = {\rm Id}_X$; thus $\Phi$ is injective. Finally, let $\varphi \in {\rm Aut} \, (X, \, R_g)$. If $\varphi$ preserves $Y$ and $Z$, then $\varphi = \Phi \bigl( (\varphi|_Y, \, \varphi|_Z), \, \overline{0} \bigr)$. Otherwise $\varphi$ interchanges $Y$ and $Z$, so $\varphi \, \sigma$ preserves them, and $\varphi = (\varphi \, \sigma) \, \sigma = \Phi \bigl( ((\varphi \sigma)|_Y, \, (\varphi \sigma)|_Z), \, \overline{1} \bigr)$. Hence $\Phi$ is also surjective.
(2) By \prref{autRg} we may identify ${\rm Aut} (X, \, R_g)$ with $H$. The second projection
$p \colon H \to \Sym{Z}$ is a group morphism. By the computation performed at the end of the
proof of \thref{bipiso}(3), applied with $(Y', \, Z', \, g') = (Y, \, Z, \, g)$, a
permutation $\beta$ lies in the image of $p$ if and only if $c_\beta (M_g) = M_g$, i.e. if and only
if $\beta \in N_g$; and for $\beta \in N_g$ the pairs $(\alpha, \, \beta) \in H$ are exactly those
with $\alpha (Y_h) = Y_{c_\beta (h)}$ for all $h\in G_g$. In particular,
$\ker p = \{(\alpha, \, {\rm Id}_Z) \, | \, \alpha (Y_h) = Y_h \,\, \forall h\}
= \prod_{h\in G_g} \Sym{Y_h}$. This gives \equref{extaut}.

For the splitting, fix for every cardinal $\kappa$ a set $I_\kappa$ of cardinality $\kappa$ (for instance $I_\kappa = \{1, \, 2, \ldots, \kappa\}$ if $\kappa$ is finite) and choose for every $h \in G_g$ a bijection
$\phi_h \colon Y_h \to I_{|Y_h|}$. For $\beta \in N_g$ we
have $|Y_{c_\beta (h)}| = |Y_h|$, so $\phi_{c_\beta (h)}$ and $\phi_h$ take values in the same set and $\phi_{c_\beta (h)}^{-1} \circ \phi_h \colon Y_h \to Y_{c_\beta (h)}$
is a bijection. Define $\alpha_\beta \in \Sym{Y}$ by
\begin{equation*}
\alpha_\beta \bigl|_{Y_h} := \phi_{c_\beta (h)}^{-1} \circ \phi_h, \qquad \text{for all} \,\, h \in G_g .
\end{equation*}
This is a well defined permutation of $Y$, as $\beta \mapsto c_\beta$ permutes $G_g$ and the fibres
partition $Y$; and $(\alpha_\beta, \, \beta) \in H$ by the description above. The map
$s \colon N_g \to H$, $s (\beta) := (\alpha_\beta, \, \beta)$, is a group morphism: for $\beta$,
$\beta' \in N_g$ and $h\in G_g$ we have
\begin{equation*}
\alpha_\beta \alpha_{\beta'} \bigl|_{Y_h}
= \alpha_\beta \bigl|_{Y_{c_{\beta'} (h)}} \circ \alpha_{\beta'} \bigl|_{Y_h}
= \phi_{c_\beta c_{\beta'} (h)}^{-1} \circ \phi_{c_{\beta'} (h)} \circ \phi_{c_{\beta'} (h)}^{-1} \circ \phi_h
= \phi_{c_{\beta \beta'} (h)}^{-1} \circ \phi_h = \alpha_{\beta\beta'} \bigl|_{Y_h},
\end{equation*}
using $c_\beta c_{\beta'} = c_{\beta\beta'}$. Since $p \circ s = {\rm Id}_{N_g}$, the sequence
\equref{extaut} is split, and \equref{autsemi} follows; the action of $\beta \in N_g$ on the kernel is the conjugation by $s(\beta) = (\alpha_\beta, \, \beta)$ in $H$, which maps $\gamma \in \Sym{Y_h}$ to $\alpha_\beta \, \gamma \, \alpha_\beta^{-1} \in \Sym{Y_{c_\beta (h)}}$; in particular, it permutes the factors $\Sym{Y_h}$ along
$h \mapsto c_\beta (h)$.

(3) Immediate from (2), since $|Y_h| = 1$ for all $h$ and the condition
$|Y_{\beta h \beta^{-1}}| = |Y_h|$ is then automatically fulfilled.
\end{proof}

Table~\ref{tab:bipcount} lists the number $\mathrm{bip} (n)$ of isomorphism classes of bipartite YB-solutions for $n \leq 8$, computed by means of \coref{countbip}; for completeness,
the split and the non-split classes are also listed separately. It shows that this number grows very rapidly with $n$.
\emph{Why $\BIP$ is a large family of unitary YB-solutions?} By \thref{autbip}(2) the automorphism group of a bipartite solution can be \emph{any} subgroup of $\Sym{Z}$ that arises as the stabilizer of a
multiset of permutations under conjugation, extended by the symmetric groups of the fibres. This is in sharp contrast with the bijective FS-solutions \equref{FSbij}, whose automorphism
groups are of the form $\Sym{A} \times {\rm Aut} (\omega)$ \cite[Theorem 2.20 (3)]{acm1}.

\begin{table}[ht]
\centering
\begin{tabular}{crrrrrrrr}
\toprule
$n = |X|$ & 1 & 2 & 3 & 4 & 5 & 6 & 7 & 8\\
\midrule
split  & 1 & 2 & 2 & 3 & 3 & 4 & 4 & 5\\
non-split & 0 & 0 & 1 & 4 & 14 & 53 & 301 & 4115\\
\midrule
$\mathrm{bip} (n)$ & 1 & 2 & 3 & 7 & 17 & 57 & 305 & 4120\\
\bottomrule
\end{tabular}
\caption{Isomorphism classes of bipartite solutions, by \coref{countbip}. The values up to $n=6$
agree with the $\BIP$ columns of Tables~\ref{tab:unrestricted} and \ref{tab:bijective} below.}
\label{tab:bipcount}
\end{table}

We close the section by recording that bipartite YB-solutions are, apart from trivial cases, very far from the solutions induced by skew
left braces. The reason is that the two families sit on opposite sides of the non-degeneracy
divide, and non-degeneracy is an isomorphism invariant.

\begin{lemma}\lelabel{ndiso}
Let $\sigma \colon (X, \, R) \to (X', \, R')$ be an isomorphism of YB-solutions. Then, $(X, \, R)$ is left-right non-degenerate, respectively right-left non-degenerate,
bijective, involutive, if and only if $(X', \, R')$ is.
\end{lemma}

\begin{proof}
The proof is a straightforward verification and is left to the reader.
\end{proof}

\begin{proposition}\prlabel{bipnd}
Let $(Y, \, Z, \, g)$ be a bipartite datum and $X = Y \sqcup Z$. Then $R_g \colon X\times X \to X\times X$ is left-right
non-degenerate if and only if $|Y| \leq 1$ and $|Z| \leq 1$. Explicitly, this happens precisely if $|X| \leq 1$, or if $|Y| = |Z| = 1$, in which case $R_g = \tau_X$.
\end{proposition}

\begin{proof}
Suppose $R_g$ is left-right non-degenerate, so that every $l_x \colon x' \mapsto x \cdot x'$ is
injective. If $y\in Y$, then $l_y$ is constant on $Y$, equal to $y$, by \equref{bimagRg}; hence
$|Y| \leq 1$. If $z\in Z$, then $l_z$ is constant on $Z$, equal to $z$; hence $|Z| \leq 1$. If
$Y$ or $Z$ is empty, the corresponding inequality holds trivially; so in all cases $|Y| \leq 1$ and $|Z| \leq 1$.

Conversely, if $|X| \leq 1$ there is nothing to prove, while if $|Y| = |Z| = 1$, then $R_g = \tau_X$, which is clearly left-right non-degenerate.
\end{proof}

\begin{corollary}\colabel{bipnotslb}
Let $(Y, \, Z, \, g)$ be a bipartite datum and $X = Y \sqcup Z$, with $|X| = |Y| + |Z| \geq 3$. Then $(X, \, R_g)$ is not
isomorphic, in the category of solutions of the braid equation, to the solution induced by any skew
left brace whatsoever.
\end{corollary}

\begin{proof}
For a skew left brace $B$, the solution $R_B$ given by \equref{slbsol} is left-right
non-degenerate \cite[Theorem 3.1]{vendramin1}. By \leref{ndiso} this property is invariant under
isomorphism of solutions, and by \prref{bipnd} the solution $R_g$ enjoys it only when
$|X| \leq 2$.
\end{proof}

\subsection*{The category \texorpdfstring{$\mathcal{LYB}$}{LYB} is not equivalent to a variety of algebras}\sselabel{lybnotvar}
In this subsection we prove that the category $\LYB$ of all Lyubashenko type solutions (\deref{lymag}), which contains $\BIP$ as a full subcategory, is not equivalent to any variety of algebras. The argument is different from the one used for $\BIP$ in \prref{BIPnotvariety}. This gives a rigorous proof of a statement made in \cite[Section 3]{ACM2}, where the proof was not included.

First we note that $\mathcal{LYB}$ is closed in $\mathcal{YB}$ under sub-solutions and direct products: if $X = \coprod_{i\in I} X_i$ is a partition as in \deref{lymag} and $Z \subseteq X$ is a sub-solution, then $Z = \coprod_{i} (Z\cap X_i)$ (the non-empty intersections) is again such a partition, while for two Lyubashenko type solutions the product of the two partitions does the job. In particular, finite limits in $\mathcal{LYB}$ (e.g. kernel pairs) are computed as in $\mathcal{YB}$, i.e. as in ${\rm Set}$. However, $\mathcal{LYB}$ is not closed under quotients:

\begin{proposition}\prlabel{lybnotvar}
Let $X = \{a, \, b, \, c, \, d\}$ and $F$, $h: X \to X$ be the maps given by
\begin{equation*}
F(a) = a, \quad F(b) = b, \quad F(c) = F(d) = c, \qquad h(a) = h(b) = b, \quad h(c) = h(d) = c.
\end{equation*}
Define $R: X\times X \to X \times X$ by $R(x, \, y) := \bigl( F(x), \, h(x) \bigl)$, for all $(x, \, y) \neq (d, \, a)$, and $R(d, \, a) := (c, \, d)$. Then:
\begin{enumerate}[(1)]
\item $(X, \, R)$ is a YB-solution of Lyubashenko type.
\item The equivalence relation $\approx$ on $X$ whose classes are $\{a\}$, $\{b, \, c\}$ and $\{d\}$ is a congruence on the associated bi-magma, and the quotient solution $(X/\!\approx, \, \overline{R})$ is not of Lyubashenko type.
\end{enumerate}
Consequently, $\mathcal{LYB}$ is not closed under quotients in $\mathcal{YB}$; in particular, by Birkhoff's theorem \cite[Theorem 3.9]{adamek_ros}, $\mathcal{LYB}$ cannot be described by identities in the two multiplications $\cdot$ and $\ast$.
\end{proposition}

\begin{proof}
The associated bi-magma is given by $x \cdot y = F(x)$, for all $x$, $y\in X$, while $x \ast y = h(x)$, for all $(x, \, y) \neq (d, \, a)$, and $d \ast a = d$. Note that $F \circ F = F$, $h \circ h = h$, $h \circ F = h$ and $F \circ h = h$.

(1) The fact that $R$ is a YB-solution follows by a straightforward routine verification, which we leave to the reader. It is of Lyubashenko type with respect to the partition $X = \{a, \, b\} \sqcup \{c, \, d\}$: on $\{a, \, b\}$ we have $R(x, \, y) = (x, \, b)$, i.e. $f_1 = {\rm Id}$ and $g_1$ is the constant map $b$, while on $\{c, \, d\}$ we have $R(x, \, y) = (c, \, c)$, i.e. $f_2 = g_2$ is the constant map $c$.

(2) The maps $F$ and $h$ are constant on $\{b, \, c\}$ modulo $\approx$ (namely $F(b) = b \approx c = F(c)$ and $h(b) = b \approx c = h(c)$), and the product $d \ast a$ involves only the singleton classes; it follows that $\approx$ is a congruence. Denote $0 := \overline{a}$, $1 := \overline{b} = \overline{c}$ and $2 := \overline{d}$. Then the quotient solution $\overline{R}$ on $\{0, \, 1, \, 2\}$ is given by
\begin{equation*}
\overline{R}(0, \, y) = (0, \, 1), \quad \overline{R}(1, \, y) = (1, \, 1), \quad \overline{R}(2, \, 0) = (1, \, 2), \quad \overline{R}(2, \, 1) = \overline{R}(2, \, 2) = (1, \, 1)
\end{equation*}
for all $y \in \{0, \, 1, \, 2\}$. Assume that $\overline{R}$ is of Lyubashenko type with respect to some partition. Since $\overline{R}(0, \, 0) = (0, \, 1)$ and $\overline{R}(2, \, 2) = (1, \, 1)$, the elements $0$ and $1$, as well as $2$ and $1$, belong to the same block; hence the partition has only one block, and $\overline{R}$ would be of the form $\overline{R}(x, \, y) = \bigl(f(x), \, g(y)\bigl)$. This is impossible since the second components of $\overline{R}(0, \, 0) = (0, \, 1)$ and $\overline{R}(2, \, 0) = (1, \, 2)$ are different.
\end{proof}

The above counterexample yields in fact a much stronger conclusion. Recall that in a variety of algebras every congruence is the kernel of the canonical projection onto the quotient algebra; in categorical terms, any equivalence relation in a variety of algebras is \emph{effective}, i.e. it is the kernel pair of its coequalizer \cite[Definition 2.5.3, Proposition 3.5.2]{Borceux2}. In other words, varieties of algebras are exact categories in the sense of Barr \cite{Barr}, \cite[Definition 2.6.1, Theorem 3.5.4]{Borceux2}. Since this property is formulated in purely categorical terms, it is preserved by equivalences of categories.

\begin{theorem}\thlabel{lybnotvar}
The category $\mathcal{LYB}$ of Lyubashenko type solutions is not equivalent to any variety of algebras.
\end{theorem}

\begin{proof}
Let $(X, \, R)$ and $\approx$ be as in \prref{lybnotvar}. The relation $\approx \, \subseteq X \times X$ is a congruence on the bi-magma associated to $R$, i.e. a sub-solution of the product solution $(X, \, R) \times (X, \, R)$. As $\mathcal{LYB}$ is closed under products and sub-solutions, $\approx$ is an object of $\mathcal{LYB}$, and together with the two projections $\approx \, \rightrightarrows X$ it is an equivalence relation on $X$ in the category $\mathcal{LYB}$. We show that this equivalence relation is not effective in $\mathcal{LYB}$. Indeed, assume that $\approx$ is the kernel pair of some morphism $\varphi \colon (X, \, R) \to (Z, \, T)$ in $\mathcal{LYB}$. Since kernel pairs in $\mathcal{LYB}$ are computed as in ${\rm Set}$, we obtain ${\rm Ker} (\varphi) \, = \, \approx$, and therefore the quotient solution $(X/\!\approx, \, \overline{R})$ is isomorphic to the sub-solution $\varphi(X)$ of $(Z, \, T)$, which is of Lyubashenko type since $\mathcal{LYB}$ is closed under sub-solutions. This contradicts \prref{lybnotvar}(2). Hence $\mathcal{LYB}$ contains a non-effective equivalence relation, and the conclusion follows from the above discussion.
\end{proof}

\begin{remark}\relabel{lybreflective}
Being closed under sub-solutions and products, $\mathcal{LYB}$ is a reflective subcategory of $\mathcal{YB}$: the reflection of a YB-solution $(Y, \, Q)$ is its quotient through the intersection $\theta$ of all congruences $\theta_i$ on $Y$ such that $Y/\theta_i$ is of Lyubashenko type (note that $Y/\theta$ embeds into $\prod_{i} Y/\theta_i$). In particular, $\mathcal{LYB}$ is complete and cocomplete. Thus, the obstruction in \thref{lybnotvar} is not the existence of (co)limits, but the failure of exactness.
\end{remark}

\section{Six families of set-theoretic Yang--Baxter solutions: a comparative count}\selabel{comparative}

The tables in this section, which classify several types of YB-solutions, were obtained by Claude (Anthropic) by means of programs written in C and Python. The author has checked them and takes full responsibility for their correctness; his role was that of asking the questions and of adapting the text to the notation and style of the paper. Moreover, most of these numbers agree with the theoretical results obtained previously in \cite{ACMv3, ACM2, acm1} and in \seref{bipartite}.

Let $X$ be a set with $|X| = n$. We count, up to isomorphism, the following six classes of YB-solutions on $X$, all of them recalled in \seref{prel}: all YB-solutions, the Lyubashenko type solutions $\LYB$, the BLS-solutions $\BLS$, the FS-solutions $\FS$, the bipartite solutions $\BIP$ and the solutions $\SLB$ arising from skew left braces. The last three families are constructed from solutions $R$ of the braid equation and transported to YB-solutions through the isomorphism \equref{iso_br_yb}, as explained in \seref{prel}, i.e. we count the YB-solutions $\tau_X \circ R$. In particular, the column $\SLB$ counts the isomorphism classes of YB-solutions of the form $\tau_X \circ R_B$, where $B$ is a skew left brace on the set $X$ and $R_B$ is the solution \equref{slbsol}; by \equref{dict}, all of them are bijective and right-left non-degenerate, which is why the $\SLB$ column is the same in Tables~\ref{tab:unrestricted}, \ref{tab:bijective} and \ref{tab:main}; in Table~\ref{tab:unitary} only the unitary ones are counted. In other words, this column records how many YB-solutions on $X$ arise, up to isomorphism, from skew left braces. The count is performed in four settings: no restriction, bijective solutions, unitary solutions and right-left non-degenerate solutions.

The Table~\ref{tab:unrestricted} gives the number of isomorphism classes of such solutions, without imposing any restriction. An asterisk marks a value which could not be computed in a reasonable time, owing to the very large amount of computation required by an exhaustive search.

\begin{table}[ht]
\centering
\begin{tabular}{crrrrrr}
\toprule
$n=|X|$ & all YB & $\LYB$ & $\BLS$ & $\FS$ & $\BIP$ & $\SLB$\\
\midrule
1 & 1    & 1   & 1    & 1  & 1 & 1\\
2 & 26   & 16  & 7    & 5  & 2 & 1\\
3 & 1045 & 408 & 55   & 14 & 3 & 1\\
4 & $*$  & $*$ & 1048 & 49 & 7 & 2\\
\bottomrule
\end{tabular}
\caption{Isomorphism classes of solutions, with no restriction imposed.}
\label{tab:unrestricted}
\end{table}

The Table~\ref{tab:bijective} lists, for each of the families, the isomorphism classes of bijective solutions. As a rule, only bijective solutions are considered in the literature \cite{BES, CV, etingof}.

\begin{table}[ht]
\centering
\begin{tabular}{crrrrrr}
\toprule
$n=|X|$ & all bijective YB & $\LYB$ & $\BLS$ & $\FS$ & $\BIP$ & $\SLB$\\
\midrule
1 & 1   & 1   & 1    & 1  & 1  & 1\\
2 & 5   & 5   & 4    & 3  & 2  & 1\\
3 & 29  & 25  & 20   & 4  & 3  & 1\\
4 & $*$ & $*$ & 207  & 8  & 7  & 2\\
5 & $*$ & $*$ & 3061 & 8  & 17 & 1\\
6 & $*$ & $*$ & $*$  & 17 & 57 & 6\\
\bottomrule
\end{tabular}
\caption{Isomorphism classes of bijective solutions.}
\label{tab:bijective}
\end{table}

Three columns of Table~\ref{tab:bijective} can also be recovered from theoretical results. By \cite[Theorem 3.28]{ACMv3}), a BLS-solution is bijective if and only if it is right-left non-degenerate; hence the $\BLS$ column coincides with the one of Table~\ref{tab:main} below. The $\FS$ column is the sequence $\sum_{d\mid n}p(d)$ of bijective FS-solutions proved before in \cite[Corollary 2.22]{acm1}. Finally, any bipartite solution is bijective (\thref{Rg}), so the $\BIP$ column is the same as in Table~\ref{tab:unrestricted}.

The Table~\ref{tab:unitary} classifies the unitary solutions, which are the most relevant ones in the theory of quantum groups. Recall that a YB-solution $r$ is unitary if and only if $R := \tau_X \circ r$ is an involutive solution of the braid equation, i.e. $R^2 = {\rm Id}_{X \times X}$.

\begin{table}[ht]
\centering
\begin{tabular}{crrrrrr}
\toprule
$n=|X|$ & all unitary YB & $\LYB$ & $\BLS$ & $\FS$ & $\BIP$ & $\SLB$\\
\midrule
1 & 1   & 1   & 1   & 1 & 1  & 1\\
2 & 3   & 3   & 2   & 2 & 2  & 1\\
3 & 8   & 8   & 5   & 2 & 3  & 1\\
4 & 64  & 53  & 17  & 3 & 7  & 2\\
5 & $*$ & $*$ & 65  & 2 & 17 & 1\\
6 & $*$ & $*$ & 323 & 4 & 57 & 2\\
\bottomrule
\end{tabular}
\caption{Isomorphism classes of unitary solutions.}
\label{tab:unitary}
\end{table}

Some columns of Table~\ref{tab:unitary} are confirmed by theoretical results. By \thref{Rg}, all bipartite solutions are unitary, so the $\BIP$ column is the same as in Table~\ref{tab:bijective}. The $\FS$ column is the sequence $\tau(n)$ of the number of divisors of $n$, since an FS-solution is involutive if and only if it is of diagonal type, and the number of isomorphism classes of such solutions on a set with $n$ elements is $\tau(n)$ \cite[Corollary 2.12]{acm1}. The unitary BLS-solutions are described in \cite[Corollary 2.7 (1)]{ACM2}.

The last table below classifies the non-degenerate solutions in the sense of \cite[Definition 1.1]{etingof}; for the sake of uniformity, we have only reformulated this notion for the Yang--Baxter equation.

\begin{table}[ht]
\centering
\begin{tabular}{crrrrrr}
\toprule
$n=|X|$ & all right-left non-deg.\ YB & $\LYB$ & $\BLS$ & $\FS$ & $\SLB$ & $\BIP$\\
\midrule
1 & 1   & 1   & 1    & 1  & 1 & 1\\
2 & 4   & 4   & 4    & 2  & 1 & 1\\
3 & 26  & 22  & 20   & 3  & 1 & 0\\
4 & 253 & 223 & 207  & 5  & 2 & 0\\
5 & $*$ & $*$ & 3061 & 7  & 1 & 0\\
6 & $*$ & $*$ & $*$  & 11 & 6 & 0\\
\bottomrule
\end{tabular}
\caption{Isomorphism classes of right-left non-degenerate solutions.}
\label{tab:main}
\end{table}

The $\FS$ column of Table~\ref{tab:main} is the sequence $p(n)$ of partition numbers. Indeed, by \equref{dict}, $\tau_X \circ R$ is right-left non-degenerate if and only if $R$ is left-right non-degenerate, and the left-right non-degenerate FS-solutions are classified in \cite[Theorem 2.19]{acm1}. The $\BIP$ column is $1, 1, 0, 0, 0, 0$ by \prref{bipnd}. We point out that at $n =3 $ there are exactly two isomorphism classes in $\LYB \setminus \BLS$. One of them is the solution $R(x, \, y) = \bigl(x, \, -(x+y) \bigr)$ on $X = \bZ/3$, the other one being $R(x, \, y) = \bigl(-(x+y), \, y \bigr)$: both are right-left non-degenerate and of diagonal type, hence of Lyubashenko type, but they are not BLS-solutions.

In all the classifications above we have used the standard solution associated with a skew left brace, given by \equref{slbsol}. Very recently, this solution has been generalized in two directions,
which we indicate below.

\begin{remarks}\label{re:noisol}
(1) Besides the solution $R_B$ given by \equref{slbsol}, a skew left brace carries a whole parametric family of solutions of the braid equation, recently introduced by Doikou and Rybo{\l}owicz \cite[Theorem 2.6]{doikou}. Let $B = (X, \, \cdot, \, \ast)$ be a skew left brace and $z \in X$ a \emph{u-distributive} element \cite[Definition 2.3]{doikou}. Then
\begin{equation}\eqlabel{dksol}
R_{B, \, z} (a, \, b) := \bigl( \sigma^z_a (b), \, \tau^z_b (a) \bigr), \qquad
\sigma^z_a (b) := (a \ast b) \cdot (a \ast z)^{-1} \cdot z, \qquad
\tau^z_b (a) := \overline{\sigma^z_a (b)} \ast a \ast b
\end{equation}
is a bijective and non-degenerate solution of the braid equation on $X$ \cite[Theorem 2.6]{doikou}. For $z = 1$ one recovers the solution $R_B$ of \equref{slbsol}, up to replacing $B$ by its opposite
skew left brace \cite[Remark 2.13]{doikou}. A computer search shows that for $|X| = 1, \dots, 6$ the numbers of isomorphism classes of solutions of the form \equref{dksol} are $1$, $1$, $1$, $3$, $1$ and $8$, respectively. Hence this is a genuine generalization: for $|X| = 4$ and $|X| = 6$ it produces new solutions, which do not arise from a skew left brace via \equref{slbsol}. For small sets, however, the gain is modest, and these numbers remain far from the total number of non-degenerate solutions recorded in Table~\ref{tab:main}.

(2) A second recent generalization is due to Catino, Mazzotta and Tieni \cite{CMT}. Here the formula \equref{slbsol} is kept, while the algebraic structure behind it is weakened: the two groups of a skew left brace are replaced by two \emph{inverse semigroups}, which leads to the notions of \emph{inverse brace} and \emph{weak brace}; for full details we refer to \cite[Theorem 4.4]{CMT} and \cite[Theorem 11]{CMMS}. In this setting the associated solutions need not be non-degenerate, but their number grows considerably: for $|X| = 1, \dots, 6$, the numbers of isomorphism classes of solutions arising from weak braces are $1$, $2$, $5$, $16$, $52$ and $212$, respectively. The numbers in (1) and (2) were obtained by an exhaustive computer search over all skew left braces, respectively all weak braces, on $X$ (in (1), $z$ running over all the elements of $X$ for which \equref{dksol} is a solution of the braid equation), followed by a reduction to isomorphism classes by means of canonical forms.
\end{remarks}

%%%%%%%%%%%%%%%%%%%%%%%%%%%%%%%%%%%
%%%%%%%%%%%%%%%%%%%%%%%%%%%%%555

\section{Concluding remarks and open questions}\label{se:conclusion}

The starting point of this paper is \cite[Remark 2.12 (2)]{CMv1}, which we recall here in a more general form. Let $X = \coprod_{i\in I} X_i$ be a partition of a
given set $X$ and assume that, for each $i \in I$, we are given a YB-solution $R_i$ on the component $X_i$. The following question is a very natural one:

{\bf Problem 1}: \emph{Can the family $(R_i)_{i\in I}$ of YB-solutions on the components $X_i$ be extended, in a natural way, to a YB-solution on the whole set $X$?}

The question is all the more natural if we keep in mind that a set on which we look for YB-solutions can be partitioned in many different ways. It is an interesting but also a
difficult one: we need a way of ``gluing'' the solutions $R_i$ in
order to obtain a YB-solution on $X$, and there is no canonical way of doing this.

In this paper we consider the simplest possible case, which, surprisingly, already produces a large number of solutions: that of a partition
$X = Y \sqcup Z$ into two subsets, with the identity maps on the two components, i.e. the most trivial solutions one can think of. We glue these solutions by means of a family of
permutations of $Z$ indexed by $Y$. This is by no means the only possible case, and in fact we propose the general gluing problem as the first step of a whole research programme.

At the opposite pole of the solutions constructed in this paper, or of those proposed in Problem~1, lies the following family of solutions, introduced in \cite[Definition 3.17 (4)]{ACMv3} as a variant of the notion of connectedness due to Lov\'asz \cite[\S 1, p.~323]{Lov}. Explicitly, a map $R \colon X\times X \to X\times X$ is called \emph{bi-connected} if there is no non-trivial partition $X = \coprod_{i\in I} X_i$ such that
$R (X_i \times X_i) \subseteq X_i \times X_i$, for all $i\in I$. Any bi-connected map $R \colon X\times X \to X\times X$ is \emph{indecomposable} in the sense of \cite[Definition 2.5]{etingof} (see \cite[Remarks 3.18 (1)]{ACMv3}), but the converse does not hold: \cite[Section 7]{rump} gives an example of an indecomposable YB-solution of diagonal type, i.e. satisfying $R(x, \, x) = (x, \, x)$ for all $x \in X$, so that the partition of $X$ into singletons is invariant.

\textbf{Problem 2}: \emph{Let $X$ be a finite set. Classify all bijective bi-connected solutions of the Yang--Baxter equation on $X$. Since the general case seems to be out of reach, a classification within certain relevant subclasses of solutions would already be of interest.}

For the solutions $R(x, \, y) = (x \cdot y, \, y \cdot x)$ associated with right P{\l}onka magmas, \cite[Corollary 3.36 (1)]{ACMv3} gives a complete answer to Problem~2: surprisingly, the number of isomorphism classes of such bi-connected solutions is the Harary number $\mathfrak{c} (n)$ \cite[Definition 3.6]{acm1}, a number well known in graph theory \cite{harary}. All bi-connected BLS-solutions are described in \cite[Theorem 3.35 (1)]{ACMv3} by means of connected pairs of commuting maps. For FS-solutions, \cite[Corollaries 3.8 and 3.9]{acm1} provide only partial answers, while the general case is treated in \cite{militaruFS}.

Finally, we believe that the following problem may be decisive for a deeper understanding of the category $\YB$:

\textbf{Problem 3}: \emph{Is there a proper and tractable subcategory $\mathcal{C} \subsetneq \YB$ such that every YB-solution is a limit or a colimit of a diagram of objects of $\mathcal{C}$?}

Unfortunately, although the subcategory $\LYB \subsetneq \YB$ is quite ``close'' to $\YB$, it is not the one sought in Problem~3 \cite[Corollary 2.14]{ACM2}.

\subsection*{Acknowledgements}
We used Claude Opus~5.5 (Anthropic) during the development and drafting of the paper to explore possible formulations and proof strategies. The tables and the numerical
results of \seref{comparative} were obtained with the help of Claude, using programs written in C and Python which enumerate the solutions exhaustively on small sets and
sort them into isomorphism classes by means of canonical forms; these programs are available from the author upon request. All statements, proofs and computational results
were independently checked by the author, who takes sole responsibility for them.

\newcommand{\Addresses}{{%
  \bigskip
  \footnotesize
  \textsc{Faculty of Mathematics and Computer Science, University of Bucharest}
  \par\nopagebreak
  \textsc{Str. Academiei 14, RO-010014 Bucharest 1, Romania}
  \par\nopagebreak
  \textit{E-mail address}: \texttt{gigel.militaru@gmail.com}
}}

\Addresses

\end{document}